\documentclass[12pt]{article}
\usepackage{amsfonts}
\usepackage{fancyhdr}
\usepackage{titlesec}
\usepackage{comment}
\usepackage{ifthen}
\usepackage{pifont}
\usepackage{stmaryrd}
\usepackage{setspace}
\usepackage{indentfirst}
\usepackage{amsmath,amssymb,amscd,bbm,amsthm,mathrsfs,dsfont}
\usepackage{color}
\usepackage{enumerate}
\usepackage{geometry}
\usepackage{enumitem}
\usepackage{graphicx}
\usepackage{float}
\usepackage{epstopdf}
\usepackage{adjustbox}
\usepackage{subcaption}
\usepackage[numbers]{natbib}
\usepackage[colorlinks=true, linkcolor=blue, urlcolor=blue, citecolor=blue]{hyperref}
\usepackage{nameref}
\usepackage{cleveref} 
\usepackage{authblk} 

\newtheorem{assumption}{Assumption}[section]
\newtheorem{theorem}{Theorem}[section]
\newtheorem{lemma}{Lemma}[section]
\newtheorem{definition}{Definition}[section]

\newtheorem{remark}{Remark}[section]
\theoremstyle{definition}

\allowdisplaybreaks[4]
\date{}
\title{Indefinite Stochastic Linear-Quadratic Optimal Control Problem for a Markov Regime-Switching Model}
\title{Indefinite Stochastic Linear-Quadratic Optimal Control Problem for a Markov Regime-Switching Model}

\author{
  Na Li\thanks{School of Mathematical Sciences, Dalian University of Technology, Dalian 116024, China. Email: {\tt lina2025@dlut.edu.cn}.}, \qquad
  Yilin Wei\thanks{Corresponding author. School of Statistics and Mathematics, Shandong University of Finance and Economics, Jinan, 250014, China. Email: {\tt 221114002@mail.sdufe.edu.cn}.}, \qquad
  Harry Zheng\thanks{Department of Mathematics, Imperial College, London, SW7 2BZ, UK. Email: {\tt h.zheng@imperial.ac.uk}. This work is supported by the National Natural Science Foundation of China (No. 12571475, 12171279).}
}
\begin{document}
\maketitle
\begin{abstract}
This paper investigates an indefinite stochastic
linear-quadratic (SLQ) control problem
with parameters subject to Markov regime-switching.
Based on the well-posedness of the SLQ problem,
we introduce a relaxed compensator
that extends SLQ control problems
from the positive definite case to the indefinite case.
We analyze the corresponding stochastic Hamiltonian system 
for both unconstrained and constrained control cases 
under the indefinite framework and 
derive the corresponding optimal open-loop controls. 
We further investigate the associated Riccati equations 
for both unconstrained and constrained control cases 
and derive the closed-loop feedback forms of optimal controls. 
We  illustrate the theoretical results with an equity-bond asset allocation problem 
 under unconstrained and non-negative control constraints. 
Numerical simulations validate the effectiveness of the theoretical framework
and demonstrate its practical value in solving complex
stochastic control problems with Markov regime-switching. \\ \par\
\end{abstract}  
\newpage
 
\section{Introduction}
The Markov regime-switching model, 
as introduced by Hamilton \cite{hamilton1989new}, 
offers a more rigorous and comprehensive framework 
for describing market dynamics 
and has been extensively applied across various domains 
within economic and financial markets. 
For example, it has been employed in interest rate modeling\cite{dahlquist2000regime}, 
investment portfolio selection \cite{guidolin2007asset,zhang2012stochastic}, 
and option pricing \cite{tian2023analysis}, among other areas. 
Furthermore, Elliott {\em et al.} \cite{elliott2005option} 
and Mehrdoust {\em et al.} \cite{mehrdoust2023two} 
demonstrate through empirical studies that Markov regime-switching models 
are effective in fitting economic and financial time series, 
especially in capturing abrupt changes. 
A key feature of the model is that 
the switching mechanism is controlled by variables 
following a first-order Markov chain. 
Markov chain is a powerful tool for analyzing regime switching 
in the economic and financial market \cite{ang2012regime}. 

Equity-bond asset allocation (EBAA) 
is a core decision in enterprise asset management, 
aiming to balance appreciation and stability 
by weighting investments across stocks and bonds. 
Traditional approaches assume stable markets 
and use fixed or slowly adjusting ratios. 
Yet real financial systems are non-stationary, 
with macroeconomic shifts, policy changes, 
and sudden shocks driving regime switches 
that alter key statistical characteristics 
such as asset returns and volatility. 
Traditional asset allocation strategies based on 
a single distribution assumption 
often fail to adapt to complete market cycles, 
leading to strategy failure and risk mismatch.
The Markov regime-switching model can
effectively make up for the deficiencies of traditional
asset allocation models, 
enabling dynamic, cycle-adaptive adjustments. 
Thus, we construct 
a dynamic system that adapts to the evolving characteristics 
of market structure, and model the problem of EBAA.
Consider the following linear system
 \begin{equation}\label{Markov_eq1}
    \begin{cases}
      \begin{aligned}
        dx_t=&[A_{t,\alpha_t}x_t+B_{t,\alpha_t}u_t]dt+[C_{t,\alpha_t}x_t+D_{t,\alpha_t}u_t]dW_t,~~~t\in [0,T], 
       \end{aligned} \\
       x_0=\xi , \alpha_0=i_0,
    \end{cases}
\end{equation}
where 
$x_{t}$ represents the holding weight of stock assets at time $t$, and 
$1-x_{t}$ represents the weight of investment in bond assets at time $t$; 
$u_{t}$ represents the intensity of stock assets weight adjustment at time $t$; 
$\xi $ represents the initial holding weight of stock assets; 
$A_{t,\alpha_t}$ represents the natural evolution coefficient of stock assets weight; 
$B_{t,\alpha_t}$ represents the influence coefficient of 
adjustment intensity on stock assets weight; 
$C_{t,\alpha_t}$ represents the random fluctuation coefficient of 
stock assets weight;
$D_{t,\alpha_t}$ represents the random shock coefficient 
of adjustment intensity. 
In addition, the cost functional is
\begin{equation}\label{Markov_eq2}
    \begin{aligned}
        \mathcal{J}(\xi ,i_0;u_{\cdot})=&\frac{1}{2}\mathbb{E}\Bigg\{\int_0^T \left[\begin{array}{c}x_t\\u_t\end{array}\right]^{\top}
                        \left[\begin{array}{c c}Q_{t,\alpha_t}&S_{t,\alpha_t}\\S^{\top}_{t,\alpha_t}&R_{t,\alpha_t}\end{array}\right]
                        \left[\begin{array}{c}x_t\\u_t\end{array}\right]dt\\
          &+x^{\top}_TG_{T, \alpha_T}x_T|x_0=\xi, \alpha_0=i_0\Bigg\},
    \end{aligned}
\end{equation}
where 
$Q_{t,\alpha_t}$ represents the risk ``penalty'' coefficient 
for the deviation of stock assets weight from the target at time $t$; 
$S_{t,\alpha_t}$ represents the ``cross-penalty'' coefficient 
of stock assets weight and the adjustment intensity at time $t$; 
$R_{t,\alpha_t}$ represents the transaction ``cost'' coefficient 
for the adjustment at time $t$; 
$G_{T,\alpha_T}$ represents the ``penalty'' coefficient for the 
terminal stock assets weight, controlling the deviation of the 
final allocation structure from the target. 
Denote $\mathcal{U}$ is the admissible control set.

\noindent\textbf{Problem (EBAA): } Given any initial condition $(\xi , i_0)\in \mathbb{R}\times \mathbb{S}$, 
find an admissible control s satisfying 
\begin{equation}\label{Markov_eq3}
   V(\xi , i_0)=\mathcal{J}(\xi ,i_0;u^*_{\cdot})=\inf_{u_{\cdot}\in \mathcal{U}}\mathcal{J}(\xi ,i_0;u_{\cdot}). 
\end{equation}
When $u^*_{\cdot}$ is not subject to any constraints, 
this problem is an EBAA with unconstrained control, 
which we refer to as Problem (EBAA-UC); 
when $u^*_{\cdot}$ is restricted to non-negative inputs, 
it is an EBAA with constrained control, 
which we refer to as Problem (EBAA-CC). 

Problem (EBAA) is a stochastic linear quadratic (SLQ) control problem, 
since \eqref{Markov_eq1} is linear and \eqref{Markov_eq2} is quadratic. 
Ji and Chizeck \cite{ji1991jump,ji2002controllability} formulate a class of
continuous-time SLQ optimal controls involving regime-switching jumps. 
Lv {\em et al.} \cite{lv2023linear} considers an SLQ leader-follower stochastic differential game 
for regime-switching diffusions with mean-field interactions. 
Wen {\em et al.} \cite{wen2023stochastic} provides a comprehensive study of SLQ optimal control problems 
in Markov regime-switching systems, where the coefficients of the state equation 
and the weighting matrices of the cost functional are random. 
Additionally, Wen {\em et al.} \cite{wen2021weak} investigates 
the open-loop and weak closed-loop solvability 
of an SLQ optimal control problem in a Markov regime-switching system. 

In the classical setting, an SLQ optimal control problem in a
Markov regime-switching system 
can be solved elegantly via Riccati equation under some mild 
conditions on the weighting coefficients — specifically, 
the positive definiteness assumption imposed on the weighting matrices
(e.g., \cite{ji1991jump,ji2002controllability,lv2023linear}). 
However, the matrices $Q_{t,\alpha_t}$, 
$R_{t,\alpha_t}$, and $G_{T,\alpha_T}$ are not required 
to satisfy the positive definite condition in Problem (EBAA). 
$Q_{t,i}$ is negative definite: it can represent the reward related to 
the value of stock assets. 
$R_{t,i}$ is negative definite: it can represent the incentive related to 
the allocation of stock assets. 
$G_{T,i}$ is negative definite: it can represent the reward 
that investors receive for the terminal asset state. 
Classical SLQ theory with Markov regime-switching 
can not solve the indefinite weight.

Chen {\em et al.} \cite{chen1998stochastic} 
points out that the control weight 
cost must be positive definite seems neither necessary for the infimum of 
the cost functional being finite nor for the existence of optimal controls. 
Since then, growing scholarly attention has been devoted to the 
study of indefinite SLQ problems (e.g., \cite{refId0,li2020indefinite,li2018indefinitedelay}). 
In particular, recent advances have primarily concentrated on indefinite SLQ control problems 
with Markov regime-switching. 
Zhang {\em et al.} \cite{zhang2021open} investigates the SLQ optimal control problem 
for Markov regime-switching systems. 
Li and Zhou \cite{li2002indefinite} examines an indefinite SLQ control problem over 
a finite time horizon, where the parameters of the problem exhibit Markovian jumps. 
Liu {\em et al.} \cite{liu2005near} investigates near-optimal controls for regime-switching SLQ problems 
with indefinite control weight costs, deriving a system of Riccati equations, proving their convergence, and constructing near-optimal controls based on the limiting system.
Overall, the existing literature demonstrates that the indefinite weight assumption is far more than just a technical generalization. 
It gives rise to more complex solution structures, 
new solvability conditions, and valuable applications 
in finance and other related fields.

The classical SLQ framework faces numerous limitations 
when directly applied to practical problems with control constraints.
These limitations mainly have two effects: first, 
the ``theoretical optimal solution" often lacks practical utility 
as it falls outside real-world feasible regions; 
second, the theoretical solution's infeasibility 
leads to a significant gap between 
practical performance and theoretical expectations. 
Adding constraints further increases complexity of SLQ problems. 
Hu and Zhou \cite{hu2005constrained} 
derives explicit feedback control 
for a class of SLQ optimal control problems 
with control variables constrained in a cone 
via two extended stochastic Riccati equations (ESREs) solutions. 
Hu {\em et al.} \cite{hu2022constrained} extends the above research to 
Markov regime-switching framework with random coefficients.
However, these studies only focus on the standard case 
(positive definite cost functional coefficients) and the singular case 
(positive semi-definite coefficients), with no systematic investigation 
of the negative definite coefficients case. 

Many researchers conduct extensive studies 
on solving indefinite SLQ problems.
Li and Zhou \cite{li2002indefinite} 
addresses the indefinite SLQ problem 
via coupled generalized Riccati equations (CGREs).
Theoretically, the optimal control can be derived by solving CGREs; 
however, the solution process, which involves Moore-Penrose pseudo-inverse 
operations, may yield infinitely many optimal solutions, 
incurs high computational cost. 
Zhang et al. \cite{zhang2021open} 
derives the cost functional using It$\hat{{\rm o}}$'s formula 
with jumps and prove the equivalence 
between the closed-loop solvability 
and the existence of a regular solution to Riccati equation, 
yet the proof process is complex and still uses the Moore-Penrose pseudo-inverse technique. 
Liu and Zhou \cite{liu2005near} addresses a system of Riccati equations. 
The authors formulate 
a relatively simplified approximate problem 
but fail to obtain the exact analytical expression 
for the optimal control strategy. 
Yu \cite{refId0} proposes an equivalent cost functional method 
for SLQ problems with random coefficients and 
negative-weighted control costs. 
\cite{li2018indefinitedelay} and \cite{li2020indefinite} 
utilize a relaxed compensator method to tackle indefinite 
SLQ problems with delay 
and indefinite stochastic mean-field-type SLQ problems, respectively.

Inspired by the relaxed compensator method, 
we analyze indefinite SLQ control problems with Markov regime-switching in this paper. 
We consider both unconstrained and constrained control cases 
and address the challenges present in traditional literature.
There are three contributions as follows: 
\begin{enumerate}[label=(\roman*)]
\item We solve the unconstrained and constrained SLQ problems 
by the relaxed compensator method, respectively. 
This approach avoids discussing the solvability of Riccati equations 
and using the Moore-Penrose pseudoinverse technique. 
In particular, if Riccati equation has a solution, 
it can be regarded as a special case of a relaxed compensator.
\item The corresponding forward-backward stochastic differential 
equations (FBSDEs) and Riccati equations under both unconstrained and 
constrained control cases are solved. Due to the fact that these equations do not 
satisfy the classical assumptions, traditional methods are no longer 
applicable. 
\item The theoretical framework is applied to Problem (EBAA) for practical validation.  
Numerical simulations under varying market states demonstrate the 
effectiveness of the proposed method, thereby confirming the framework's practical applicability.
\end{enumerate}
The remainder of this paper is organized as follows. 
Section \ref{sec: 2} introduces notations and formulates an indefinite SLQ problem with Markov regime-switching. 
Section \ref{sec: Hamiltonian} analyzes the corresponding stochastic Hamiltonian system for both unconstrained and constrained control cases.
Section \ref{sec: Riccati} focuses on the related Riccati equation for both control cases.
Section \ref{sec: 5} studies Problem (EBAA) in unconstrained control and non-negative control constraints, and verifies the derived theoretical results via numerical simulations.
Section \ref{sec: 6} concludes the paper.

\section{Problem Formulation and Preliminaries}\label{sec: 2}
In this section, we formulate an indefinite SLQ problem with Markov regime-switching. 
Let $\mathbb{R}^n$ denote the $n$-dimensional Euclidean space, 
and let $\mathbb{R}^{n \times m}$ represent the space of all $n \times m$ matrices. 
We denote $\mathbf{0}$ as zero matrices of appropriate dimensions. 
The usual norm is denoted by $\lvert \cdot \rvert$, 
and the inner product in $\mathbb{R}^n$ is represented by $\left\langle \cdot, \cdot \right\rangle$. 
For a given vector or matrix $A$, the transpose is denoted by $A^\top$, 
and $A^{-1}$ denotes the inverse of $A$ when it exists. 
For any real number, we define $x^+ :=\max\{x,0\}$ and $x^- :=\max\{-x,0\}$. 
We consider a finite time horizon $[0, T]$ with a fixed $T > 0$. 
Let $(\Omega, \mathcal{F}, \mathbb{P}, \mathbb{F})$ be a complete filtered probability space. 
Within this space, we define two mutually independent stochastic processes: 
a one-dimensional Brownian motion $\{W_{t},~ t\in [0, T]\}$, 
and an irreducible homogeneous continuous-time Markov chain $\{\alpha_{t},~ t\in [0, T]\}$. 
$\mathcal{N}$ denotes the set of all the $\mathbb{P}$-null sets. 
For each $t \in [0, T]$, the filtration $\mathcal{F}^\alpha_t$ is generated by 
$\sigma\left\{\alpha_{s},~s\in [0, t]\right\}\vee\mathcal{N}$ and 
$\mathcal{F}_t$ is generated by 
$\sigma\left\{\alpha_{s}, W_{s},~s\in [0, t]\right\}\vee\mathcal{N}$. 
The filtration $\mathbb{F}$ is then defined as $\mathbb{F} := \{ \mathcal{F}_t,~ t \in [0, T] \}$. 
The conditional expectation with respect to $\mathcal{F}^\alpha_t$ is denoted by $\mathbb{E}[\cdot|\mathcal{F}^\alpha_t]$.

We assume that the Markov chain takes values in a finite state space $\mathbb{S}=\left\{1, \dots, d\right\}$ 
and initiates in a fixed state $i_0\in \mathbb{S}$ such that $\alpha_0=i_0$. 
The generator of the Markov chain $\alpha_t$ is denoted by $\mathbb{G}$, 
a $d \times d$ matrix given by $\mathbb{G} = (\pi_{i,j})_{i,j=1}^d$, 
where $\pi_{i,j}$ represents the transition rate or jump intensity of the Markov chain from state $i$ to state $j$. 
The transition probabilities of the Markov chain are described by the following expression 
$$\mathbb{P}\{\alpha_{t+\Delta t}=j|\alpha_t=i\}=\left\{\begin{array}{ll}\pi_{i,j}\Delta t+o(\Delta t),&\mathrm{if~}i\neq j,\\1+\pi_{i,i}\Delta t+o(\Delta t),&\mathrm{else},\end{array}\right.$$
where $\pi_{i,j}\ge 0$ for $i\neq j$ and $\pi_{i,i}=-\sum_{j\neq i}\pi_{i,j}$. 
Let $\mathcal{X}$ denote the indicator function. For each distinct pair of states $(i, j)$ in the state space $\mathbb{S}$, 
we associate a counting process defined by
$$M_{i,j}(t):=\sum_{0<s\leq t}\mathcal{X} _{\{\alpha_{s_-}=i\}}\mathcal{X} _{\{\alpha_{s}=j\}}, ~~~\forall t\in[0, T],~ i,j\in \mathbb{S}.$$ 
The process $M_{i,j}(t)$ counts the number of times the Markov chain $\alpha_t$ transitions from state $i$ to state $j$ by time $t$. 
If we compensate $M_{i,j}(t)$ by $\int_0^t\pi_{i,j}\mathcal{X} _{\{\alpha_{s_-}=i\}}ds$, then the resulting process
$$\widetilde{M}_{i,j}(t):=M_{i,j}(t)-\int_0^t\pi_{i,j}\mathcal{X} _{\{\alpha_{s_-}=i\}}ds$$
is a purely discontinuous square-integrable martingale with an initial value of zero (see Rogers and Williams \cite{rogers2000diffusions}). 

Moreover, we denote the sets of matrices as follows:
\begin{itemize} 
    \item $\mathcal{S}^n:$ the space of all $n\times n$ symmetric matrices.
    \item $\mathcal{S}^n_+:$ the subspace of all positive semi-definite matrices of $\mathcal{S}^n$.
    \item $\mathcal{S}^n_{++}:$ the subspace of all positive definite matrices of $\mathcal{S}^n$.
\end{itemize} 

For Euclidean space $\mathscr{H}$, here are some spaces that are carried out throughout this paper. 
\begin{itemize} 
    \item $C(s,l;\mathscr{H}):$ the space of all continuous functions $\{g(t),t\in [s,l]\}$. 
    \item $L^2(s,l;\mathscr{H}):$ the space of $\mathscr{H}$-valued deterministic functions $\{g(t),t\in [s,l]\}$ such that 
     $\int ^l_s\lvert g(t)\rvert^2dt<\infty$. 
    \item $L^\infty(s,l;\mathscr{H}):$ the space of $\mathscr{H}$-valued, unifomly bounded deterministic functions $\{g(t),t\in [s,l]\}$. 
    \item $L^2_{\mathbb{F}}(s,l;\mathscr{H}):$ the space of $\mathscr{H}$-valued, $\mathbb{F}$-progressively measurable stochastic processes $\{g(t,\omega), (t,\omega)\in [s,l]\times \Omega\}$ such that $\mathbb{E}\int^l_s\lvert g(t,\omega)\rvert^2dt<\infty$. 
    \item $L^{\infty}_{\mathbb{F}}(s,l;\mathscr{H}):$ the space of $\mathscr{H}$-valued, $\mathbb{F}$-progressively measurable bounded stochastic processes. 
    \item $\mathscr{M}^2_{\mathbb{F}}(s,l;\mathscr{H}):$ the space of $\mathbb{F}$-adapted and $c\grave{a}dl\grave{a}g$ stochastic processes $\{g(t,\omega,\alpha_t),$\\$(t,\omega,\alpha_t)\in [s,l]\times\Omega\times\mathbb{S}\}$ 
    such that $\mathbb{E}[\sup_{s\le t\le l}\lvert g(t,\omega,\alpha_t)\rvert^2]< \infty$. 
\end{itemize} 

 We consider a high dimensional controlled dynamic system with Markov regime-switching 
 with the form of $\eqref{Markov_eq1}$, 
where $A_{\cdot,\alpha_{\cdot}}$, $C_{\cdot,\alpha_{\cdot}}\in L^{\infty}(0, T;\mathbb{R}^{n\times n})$; 
$B_{\cdot,\alpha_{\cdot}}$, $D_{\cdot,\alpha_{\cdot}}\in L^{\infty}(0, T;\mathbb{R}^{n\times {m}})$; 
$\xi \in\mathbb{R}^n$; 
and $i_0\in \mathbb{S}$. 
$A_{t,\alpha_t}=A_{t, i}$, $B_{t,\alpha_t}=B_{t, i}$, 
$C_{t,\alpha_t}=C_{t, i}$, and $D_{t,\alpha_t}=D_{t, i}$ when $\alpha_t=i$ for each $i\in \mathbb{S}$. 
Let $\varGamma=\mathbb{R}^{m}_+$ be a given closed cone, 
i.e., $\varGamma$ is closed, and if $u \in \varGamma$, 
then $\kappa u \in \varGamma$, for all $\kappa>0$. 
This set $\varGamma$ serves as the constraint set for control values. 
The corresponding constrained admissible control set is defined as 
$\mathcal{U}^c_{ad}=\{u_\cdot \in L^2_{\mathbb{F}}(0, T; \mathbb{R}^{m}_+)|u_t\in \mathbb{R}^{m}_+\}$. 
If $\varGamma = \mathbb{R}^{m}$, 
the set $\varGamma$ serves as the unconstrained set for control values. 
In this case, the corresponding unconstrained admissible control set 
is denoted by 
$\mathcal{U}_{ad}=\{u_\cdot\in L^2_{\mathbb{F}}(0, T; \mathbb{R}^{m})|u_t \in\mathbb{R}^{m}\}$.
For notational simplicity in the general problem formulation, 
we shall use $\mathcal{U}$ to represent the admissible control set, 
which can be either $\mathcal{U}_{ad}$ or $\mathcal{U}^c_{ad}$ 
depending on the context.
Each element $u_{\cdot}$ in admissible control sets is called an admissible control. 
For a given initial condition $(\xi, i_0)$ and an admissible control $u_{\cdot}$, 
the associated cost functional is given by $\eqref{Markov_eq2}$,
where $Q_{\cdot, \alpha_{\cdot}} \in L^{\infty}(0, T; \mathcal{S}^n)$, $S_{\cdot, \alpha_{\cdot}} \in L^{\infty}(0, T; \mathbb{R}^{n \times m})$, $R_{\cdot, \alpha_{\cdot}} \in L^{\infty}(0, T; \mathcal{S}^m)$, and $G_{\cdot, \alpha_{\cdot}} \in \mathcal{S}^n$. 
When $\alpha_t = i$ (for any $i \in \mathbb{S}$), the matrices $Q_{t,\alpha_t}$, $S_{t,\alpha_t}$, and $R_{t,\alpha_t}$ 
are specified as $Q_{t, i}$, $S_{t, i}$, and $R_{t, i}$, respectively. 
Similarly, $G_{T, \alpha_T}$ is denoted by $G_{T, i}$ when $\alpha_T = i$ for each $i \in \mathbb{S}$.
We now formulate the SLQ control problem with Markov regime-switching (MRS) as follows.

\noindent\textbf{Problem (SLQ-MRS): }Given any initial condition $(\xi , i_0)\in \mathbb{R}^n\times \mathbb{S}$, 
find an admissible control $u^*_{\cdot}\in \mathcal{U}$ satisfying $\eqref{Markov_eq3}$.

If such an admissible control $u^*_{\cdot}$ exists, then $u^*_{\cdot}$ 
is called the optimal control for Problem (SLQ-MRS), $x^*_{\cdot}$ 
is called the corresponding optimal state trajectory, and 
$(x^*_{\cdot}, u^*_{\cdot})$ is an optimal pair (with respect to the initial condition $(\xi , i_0)$). 
Additionally, Problem (SLQ-MRS) is well-posed if 
$ V(\xi , i_0)>-\infty$, for all $(\xi , i_0)\in \mathbb{R}^n\times \mathbb{S}.$
When $u^*_{\cdot}\in \mathcal{U}_{ad}$, 
this problem is an SLQ-MRS with unconstrained control, 
which we refer to as Problem (SLQ-MRS-UC); 
when $u^*_{\cdot}\in \mathcal{U}^c_{ad}$, 
it is an SLQ-MRS with constrained control, 
which we refer to as Problem (SLQ-MRS-CC). 
 
To proceed, we introduce the positive definite condition and Schur's lemma, which will be key tools in the subsequent analysis. 
\begin{assumption}\label{Asm 2.1}{\rm (Positive definite condition)}  
    $$\left[Q_{\cdot,\alpha_{\cdot}}\ S_{\cdot,\alpha_{\cdot}};S^{\top}_{\cdot,\alpha_{\cdot}}\ R_{\cdot,\alpha_{\cdot}}\right]\in L^\infty(0, T; \mathcal{S}^{n+m}_+), \ G_{\cdot,\alpha_{\cdot}}\in \mathcal{S}^n_+,\ and\  R_{\cdot,\alpha_{\cdot}}\in L^\infty(0, T; \mathcal{S}^{m}_{++}).$$
\end{assumption}
\begin{lemma}\label{lem 2.1}{\rm (Schur's lemma)}
    Let $Q=Q^{\top}$, $S$, and $R=R^{\top}$
    be matrices of appropriate dimensions. The following conditions are equivalent
    \begin{enumerate}[label=(\roman*)]
        \item $Q-SR^{-1} S^{\top}\geq {\bf{0}},~R>{\bf{0}};$
        \item $\left[\begin{array}{c c}Q&S\\S^{\top}&R\end{array}\right]\geq {\bf{0}},~R>{\bf{0}}.$
    \end{enumerate}
\end{lemma}
This paper focuses on the indefinite case. 
Namely, $Q_{t,\alpha_t}$, $S_{t,\alpha_t}$, $R_{t,\alpha_t}$, and $G_{T,\alpha_T}$ are all possibly indefinite. 
In this context, we aim to establish the well-posedness of 
Problem (SLQ-MRS) by introducing a relaxed compensator. 
Specifically, we define the following set of diffusion processes 
$$
\begin{aligned}
        \varUpsilon:=&\Bigg\{K_{\cdot, \alpha_{\cdot}}\in L^{\infty}_{\mathbb{F}}(0, T; \mathcal{S}^n)|K_{t,\alpha_t}=K_{0,\alpha_0}+\int_{0}^{t}\Theta _{s,\alpha_s}ds+\int_{0}^{t}\sum_{i,j=1}^{d}\pi_{i,j}(K_{s,j}-K_{s,i})ds\\
        &+\int_{0}^{t}\sum_{i,j=1}^{d}(K_{s,j}-K_{s,i})d\widetilde{M}_{i,j}(s), ~~~\forall s\in [0, t]
                        \Bigg\},
\end{aligned}$$
where $\Theta _{\cdot, \alpha_{\cdot}}\in L^\infty(0, T; \mathcal{S}^n)$ and $K_{0,\alpha_0}\in \mathcal{S}^n$. 
As pointed by Zhang {\em et al.} \cite{zhang2021open}, a solution of the process $K_{t,\alpha_t}$ exists. 

Let 
\begin{equation}
    \begin{aligned}
        \mathcal{J}^K(\xi ,i_0;u_{\cdot})=&\frac{1}{2}\mathbb{E}\Bigg\{\int_0^T \left[\begin{array}{c}x_t\\u_t\end{array}\right]^{\top}
                    \left[\begin{array}{c c}\mathcal{Q}^K_{t,\alpha_t}&\mathcal{S}^K_{t,\alpha_t}\\{(\mathcal{S}^K_{t,\alpha_t})}^{\top}&\mathcal{R}^K_{t,\alpha_t}\end{array}\right]
                    \left[\begin{array}{c}x_t\\u_t\end{array}\right]dt\\
                    &+x^{\top}_T\mathcal{G}^K_{T,\alpha_T}x_T|x_0=\xi, \alpha_0=i_0\Bigg\},
    \end{aligned}
    \nonumber
\end{equation}
where 
\begin{equation}
\begin{aligned}
    &\mathcal{Q}^K_{t,\alpha_t}=Q_{t,\alpha_t}+\Theta _{t,\alpha_t}+K_{t,\alpha_t}A_{t,\alpha_t}+A^{\top}_{t,\alpha_t}K_{t,\alpha_t}\\
    &\quad\qquad+C^{\top}_{t,\alpha_t}K_{t,\alpha_t}C_{t,\alpha_t}+\sum_{j=1}^{d}\pi_{\alpha_{t_-},j}(K_{t, j}-K_{t,\alpha_{t_-}}),\\
    &\mathcal{S}^K_{t,\alpha_t}=S_{t,\alpha_t}+K_{t,\alpha_t}B_{t,\alpha_t}+C^{\top}_{t,\alpha_t}K_{t,\alpha_t}D_{t,\alpha_t}, \\
    &\mathcal{R}^K_{t,\alpha_t}=R_{t,\alpha_t}+D^{\top}_{t,\alpha_t}K_{t,\alpha_t}D_{t,\alpha_t}, \\
    &\mathcal{G}^K_{T,\alpha_t}=G_{T,\alpha_T}-K_{T,\alpha_T}. 
\end{aligned}
\nonumber
\end{equation}
Now, we define a relaxed compensator as follows. 
\begin{definition}\label{def 2.1}
   If there exists $K_{\cdot, \alpha_{\cdot}}\in \varUpsilon $ such that $(\mathcal{Q}^K_{\cdot,\alpha_{\cdot}}, \mathcal{S}^K_{\cdot,\alpha_{\cdot}}, \mathcal{R}^K_{\cdot,\alpha_{\cdot}}, \mathcal{G}^K_{\cdot,\alpha_{\cdot}})$ 
    satisfies Assumption \ref{Asm 2.1}, then $K_{\cdot, \alpha_{\cdot}}$ is called a relaxed compensator of Problem (SLQ-MRS). 
\end{definition}
The existence of a relaxed compensator is crucial in solving the indeﬁnite problem. 
Subsequently, we establish a sufficient and necessary condition 
for the existence of a relaxed compensator, which can be rigorously verified through Schur's lemma. 
The detailed proof is omitted here for brevity. 
 \begin{lemma}\label{Prop 2.1}
  $K_{\cdot, \alpha_{\cdot}}\in \varUpsilon $ is a relaxed compensator of Problem (SLQ-MRS) 
    if and only if the following conditions hold
       \begin{enumerate}[label=(\roman*)]
        \item $\mathcal{Q}^K_{t,\alpha_t}-\mathcal{S}^K_{t,\alpha_t}(\mathcal{R}^K_{t,\alpha_t})^{-1} (\mathcal{S}^K_{t,\alpha_t})^{\top}\geq {\bf{0}},~~~t\in [0,T];$
        \item $\mathcal{R}^K_{t,\alpha_t}>{\bf{0}},~~~t\in [0,T]$;\ $\mathcal{G}^K_{T,\alpha_T}\geq \bf{0}$.
    \end{enumerate}
\end{lemma}

\begin{theorem}\label{Thm 2.1}
If there exists a relaxed compensator $K_{\cdot, \alpha_{\cdot}}\in \varUpsilon$, then Problem (SLQ-MRS) is well-posed. 
\end{theorem}
\begin{proof}
    For any $K_{\cdot, \alpha_{\cdot}}\in \varUpsilon$, applying It$\hat{{\rm o}}$'s formula to $\langle K_{t,\alpha_t}x_t, x_t \rangle $ on the interval $[0, T]$ and taking conditional expectation, 
    we get
    \begin{equation}\label{Markov_eq6}
        \begin{aligned}
            &\mathbb{E}[\langle K_{T,\alpha_T}x_T, x_T \rangle-\langle K_{0, \alpha_0}x_0 , x_0  \rangle|x_0=\xi, \alpha_0=i_0]\\
            &=\mathbb{E}[\langle (G_{T,\alpha_T}-\mathcal{G}^K_{T,\alpha_T})x_T, x_T - K_{0, \alpha_0}x_0 , x_0  \rangle|x_0=\xi, \alpha_0=i_0]\\
            &=\mathbb{E}\Big\{\int_{0}^{T}\big\{\langle (\mathcal{Q}^K_{t,\alpha_t}-Q_{t,\alpha_t})x_t, x_t \rangle+2\langle (\mathcal{S}^K_{t,\alpha_t}-S_{t,\alpha_t})u_t, x_t \rangle\\
            &\quad+\langle (\mathcal{R}^K_{t,\alpha_t}-R_{t,\alpha_t})u_t, u_t \rangle\big\}dt|x_0=\xi, \alpha_0=i_0\Big\}. \\
        \end{aligned}
    \end{equation} 
Combining \eqref{Markov_eq6} with the expression of \eqref{Markov_eq2}, we obtain 
    \begin{equation}\label{Markov_eq7}
\mathcal{J}(\xi ,i_0;u_{\cdot})=\mathcal{J}^K(\xi ,i_0;u_{\cdot})+\frac{1}{2}\langle K_{0, i_0}x_0 , x_0  \rangle. 
\end{equation}
According to the definition of the relaxed compensator, 
    $$\left[\begin{array}{c c}\mathcal{Q}^K_{t,\alpha_t}&\mathcal{S}^K_{t,\alpha_t}\\{(\mathcal{S}^K_{t,\alpha_t})}^{\top}&\mathcal{R}^K_{t,\alpha_t}\end{array}\right]\in L^\infty(0, T; \mathcal{S}^{n+m}_+) \ {\rm and}\  \mathcal{G}^K_{T,\alpha_T}\in \mathcal{S}^n_+, $$
    then $\mathcal{J}(\xi ,i_0;u_{\cdot})\geq \frac{1}{2}\langle K_{0, i_0}\xi , \xi  \rangle>-\infty, $
    which implies that Problem (SLQ-MRS) is well-posed.
\end{proof}
Equation $\eqref{Markov_eq7}$ implies that cost functionals 
$\mathcal{J}(\xi ,i_0;u_{\cdot})$ and $\mathcal{J}^K(\xi ,i_0;u_{\cdot})$ are equivalent. 
Solving Problem (SLQ-MRS) is equivalent to solving the system described by equations $\eqref{Markov_eq1}$ and $\eqref{Markov_eq7}$. 
Both problems share the same optimal control and state processes, yet yield different optimal values. 
The equivalence allows the original indefinite Problem (SLQ-MRS) 
to be effectively transformed into a positive definite one.
It is crucial to note that this conclusion is independent of 
whether the control is constrained or not; 
therefore, it is equally applicable to both Problem (SLQ-MRS-UC) 
and Problem (SLQ-MRS-CC). 
\begin{remark}\label{rem 2.1}
In contrast to the classical methods for solving indefinite SLQ problems 
(see \cite{li2002indefinite,liu2005near,zhang2021open}), 
the relaxed compensator method demonstrates two remarkable advantages: 
(i) The classical solution paths usually require directly solving CGREs, 
but these equations are still difficult to obtain analytical solutions in practice. 
The construction of the relaxed compensator 
is not restricted by strict conditions, 
and does not rely on the preconditions of 
the solvability of Riccati equations.
Therefore, it has greater flexibility and ingeniously avoids 
the complex derivation process of the solvability conditions 
of Riccati equations. 
(ii) The classical methods depending on Moore-Penrose pseudo-inverse operations need 
to handle the indefinite problems, while our method does not rely on 
these computationally intensive techniques.
\end{remark}
\section{\texorpdfstring{Stochastic Hamiltonian System with Markov \\Regime-Switching}{Stochastic Hamiltonian System with Markov Regime-Switching}}\label{sec: Hamiltonian}
In this section, we present a study of the stochastic Hamiltonian system 
for indefinite SLQ optimal control problems with Markov regime-switching. 
\subsection{\texorpdfstring{Stochastic Hamiltonian System with Markov \\Regime-Switching for Unconstrained Control}
{Stochastic Hamiltonian System with Markov Regime-Switching for Unconstrained Control}}\label{sec: Hamiltonian.1}

We employ an invertible linear transformation incorporating the relaxed compensator to 
address the corresponding stochastic Hamiltonian system for unconstrained control under the indefinite condition. 
And then, we derive an explicit representation of the unconstrained optimal control in an open-loop form for Problem (SLQ-MRS-UC).

The stochastic Hamiltonian system for Problem (SLQ-MRS-UC) 
is given by Stochastic maximum principle as follows, 
\begin{equation}\label{Markov_eq9}
    \begin{cases}
        \begin{aligned}
            \textbf{0}=R_{t,\alpha_t}u^*_t+S^{\top}_{t,\alpha_t}x^*_t+B^{\top}_{t,\alpha_t}q^*_t+D^{\top}_{t,\alpha_t}z^*_t,~~~t\in [0,T], \\
        \end{aligned}\\
        \begin{aligned}
            dx^*_t=[A_{t,\alpha_t}x^*_t+B_{t,\alpha_t}u^*_t]dt+[C_{t,\alpha_t}x^*_t+D_{t,\alpha_t}u^*_t]dW_t,~~~t\in [0,T], \\
        \end{aligned}\\
        \begin{aligned}
            -dq^*_t=[A^{\top}_{t,\alpha_t}q^*_t+C^{\top}_{t,\alpha_t}z^*_t+Q_{t,\alpha_t}x^*_t+S_{t,\alpha_t}u^*_t]dt
            -z^*_tdW_t-\eta^*_t\bullet d\widetilde{M}_t,~~~t\in [0,T], \\
        \end{aligned}\\
            x^*_0=\xi ,~\alpha_0=i_0,~q^*_T=G_{T,\alpha_T}x^*_T,\\
    \end{cases}
\end{equation}
where $\eta_{t}\bullet d\widetilde{M}_{t}=\sum_{i,j=1}^{d}\eta_{i,j}(t) d\widetilde{M}_{i,j}(t)$. 

Denote 
        $$\begin{aligned}
    &H(t, x_{t}, u_{t}, \alpha_t, q_{t}, z_{t};Q_{t,\alpha_t},S_{t,\alpha_t},R_{t,\alpha_t})\\
    &:=\frac{1}{2}\left[(x_t)^{\top}Q_{t,\alpha_t}x_t+2(x_t)^{\top}S_{t,\alpha_t}u_t+(u_t)^{\top}R_{t,\alpha_t}u_t\right]\\
    &\ \quad +(A_{t,\alpha_t}x_t+B_{t,\alpha_t}u_t)^{\top}q_t+(C_{t,\alpha_t}x_t+D_{t,\alpha_t}u_t)^{\top}z_t.
\end{aligned}$$
\begin{theorem}\label{Thm 3.1}
  If there exists a relaxed compensator $K_{\cdot, \alpha_{\cdot}}\in\varUpsilon $, 
    the stochastic Hamiltonian system of Problem (SLQ-MRS-UC) admits a unique solution 
    $(x^*_{\cdot}, u^*_{\cdot}, q^*_{\cdot}, z^*_{\cdot},\eta^* _{\cdot})\in L^2_{\mathbb{F}}(0, T; \mathbb{R}^n)\times \mathcal{U}_{ad}
    \times L^2_{\mathbb{F}}(0, T; \mathbb{R}^n)\times L^2_{\mathbb{F}}(0, T; \mathbb{R}^n)
    \times\mathscr{M}^2_{\mathbb{F}}(0, T; \mathbb{R}^n)$, where $(x^*_{\cdot},u^*_{\cdot})$ is the unique optimal pair for Problem (SLQ-MRS-UC). 
        \end{theorem}
\begin{proof}
    For a given $K_{\cdot, \alpha_{\cdot}}\in\varUpsilon $, 
we introduce the following new stochastic Hamiltonian system:
    \begin{equation}\label{Markov_eq13}
        \begin{cases}
                   \begin{aligned}
                &dx^{K*}_t=[A_{t, \alpha_t}x^{K*}_t+B_{t, \alpha_t}u^{K*}_t]dt+[C_{t, \alpha_t}x^{K*}_t+D_{t, \alpha_t}u^{K*}_t]dW_t,~~~t\in [0,T], \\
                &-dq^{K*}_t=[A^{\top}_{t, \alpha_t}q^{K*}_t+C^{\top}_{t, \alpha_t}z^{K*}_t+\mathcal{Q}^K_{t,\alpha_t}x^{K*}_t+\mathcal{S}^K_{t,\alpha_t}u^{K*}_t]dt\\
                &\ \qquad\qquad-z^{K*}_tdW_t-\eta^{K*}_t\bullet d\widetilde{M}_t,~~~t\in [0,T], \\
            \end{aligned}\\
            x^{K*}_0=\xi ,~\alpha_0=i_0,~q^{K*}_T=\mathcal{G}^K_{T,\alpha_T}x^{K*}_T,\\
        \end{cases}
    \end{equation}
    and
 \begin{equation}\label{eq10}
 {\bf{0}}=\mathcal{R}^K_{t,\alpha_t}u^{K*}_t+{(\mathcal{S}^K_{t,\alpha_t})}^{\top}x^{K*}_t+B^{\top}_{t, \alpha_t}q^{K*}_t+D^{\top}_{t, \alpha_t}z^{K*}_t.
\end{equation}            
For any $(x^K_{\cdot}, u^K_{\cdot})\in L^2_{\mathbb{F}}(0, T; \mathbb{R}^n)\times \mathcal{U}_{ad}$,  
the adapted solution to the first equation in $\eqref{Markov_eq13}$ admits a unique solution, see Li and Zheng \cite{li2015weak}. 
    According to \eqref{Markov_eq9} and \eqref{Markov_eq13}, we get the following equivalence relationships
    \begin{equation}\label{Markov_eq16}
        \begin{aligned}
            &x^*_t=x^{K*}_t, u^*_t=u^{K*}_t, q^*_t=q^{K*}_t+K_{t, \alpha_t}x^*_t,\\
            &z^*_t=z^{K*}_t+K_{t, \alpha_t}D_{t, \alpha_t}u^*_t+K_{t, \alpha_t}C_{t, \alpha_t}x^*_t,\\
            &\eta ^*_{i,j}=\eta ^{K*}_{i,j}+[K_{t, j}-K_{t, i}]x^*_t,\\
        \end{aligned}
    \end{equation}
    which solves the stochastic Hamiltonian system \eqref{Markov_eq9} of Problem (SLQ-MRS-UC). 
    If $K_{\cdot, \alpha_{\cdot}}\in\varUpsilon$ is the relaxed compensator, 
    we solve $u^{K*}_{\cdot}$ in the following open-loop form 
    from (\ref{eq10})  by Theorem 3.4 in Tao and Wu \cite{tao2012maximum}
    \begin{equation}\label{Markov_eq14}
            \begin{aligned}
                u^{K*}_t=-{(\mathcal{R}^K_{t, \alpha_t})}^{-1}\bigl[{(\mathcal{S}^K_{t, \alpha_t})}^\top x^{K*}_t+B_{t, \alpha_t}^\top q^{K*}_t+D^\top_{t, \alpha_t}z^{K*}_t\bigr]. \\
            \end{aligned}
        \end{equation} 
     Substituting $\eqref{Markov_eq14}$ into $\eqref{Markov_eq13}$, 
        then $\eqref{Markov_eq13}$ is reduced to a fully coupled linear FBSDE 
        with Markov regime-switching,
       which          admits a unique solution 
        $(x^{K*}_{\cdot}, q^{K*}_{\cdot}, z^{K*}_{\cdot}, \eta^{K*} _{\cdot})
        \in L^2_{\mathbb{F}}(0, \\T; \mathbb{R}^n)
        \times L^2_{\mathbb{F}}(0, T; \mathbb{R}^n)
        \times L^2_{\mathbb{F}}(0, T; \mathbb{R}^n)
        \times\mathscr{M}^2_{\mathbb{F}}(0, T; \mathbb{R}^n)$ and 
 is a special case of the relevant arguments 
        in the literature of Tao and Wu \cite{tao2012maximum}. 
       Then stochastic Hamiltonian system \eqref{Markov_eq13}-\eqref{eq10} admits a unique solution 
        $(x^{K*}_{\cdot}, u^{K*}_{\cdot}, q^{K*}_{\cdot}, z^{K*}_{\cdot}, \eta^{K*} _{\cdot})
        \in L^2_{\mathbb{F}}(0, T; \mathbb{R}^n)
        \times \mathcal{U}_{ad}
        \times L^2_{\mathbb{F}}(0, T; \mathbb{R}^n)
        \times L^2_{\mathbb{F}}(0, T; \mathbb{R}^n)
        \times\mathscr{M}^2_{\mathbb{F}}(0, T; \mathbb{R}^n)$.
    Since $\eqref{Markov_eq16}$ is invertible, 
        it follows from the existence and uniqueness of $\eqref{Markov_eq13}$ that 
        the solution of \eqref{Markov_eq9} exists and is unique. 
        Thus, the stochastic Hamiltonian system \eqref{Markov_eq9} of Problem (SLQ-MRS-UC) admits a unique solution 
    $(x^*_{\cdot}, u^*_{\cdot}, q^*_{\cdot}, z^*_{\cdot},\eta^* _{\cdot})\in L^2_{\mathbb{F}}(0, T; \mathbb{R}^n)\times \mathcal{U}_{ad}
    \times L^2_{\mathbb{F}}(0, T; \mathbb{R}^n)\times L^2_{\mathbb{F}}(0, T; \mathbb{R}^n)
    \times\mathscr{M}^2_{\mathbb{F}}(0, T; \mathbb{R}^n)$. 
    Additionally, if another relaxed compensator $\widetilde{K}_{\cdot, \alpha_{\cdot}}\in\varUpsilon$ exists, the Hamiltonian system \eqref{Markov_eq13}-\eqref{eq10} with $\widetilde{K}_{\cdot, \alpha_{\cdot}}$ admits a unique solution, 
        so the Hamiltonian system \eqref{Markov_eq13}-\eqref{eq10} with $K_{\cdot, \alpha_{\cdot}}$ admits a unique solution either. 
        That is to say the existence and uniqueness of the Hamiltonian system $\eqref{Markov_eq13}$ with $\widetilde{K}_{\cdot, \alpha_{\cdot}}$ and $K_{\cdot, \alpha_{\cdot}}$ are equivalent, 
        i.e. $(x^{K*}_{\cdot}, u^{K*}_{\cdot})=(x^{\widetilde{K}*}_{\cdot}, u^{\widetilde{K}*}_{\cdot})=(x^*_{\cdot}, u^*_{\cdot})$. 

        From Theorem \ref{Thm 2.1}, we prove that $u^*_{\cdot}$ is the unique optimal control for Problem (SLQ-MRS-UC) 
        by proving that $u^{K*}_{\cdot}$ is the unique optimal control for the system described by equations $\eqref{Markov_eq1}$ and $\eqref{Markov_eq7}$.

        Under Assumption \ref{Asm 2.1}, $\mathcal{J}^K(u^K_{\cdot})$ is convex, so
\begin{equation}
    \begin{aligned}
        &\mathcal{J}^K(u^{K*}_{\cdot})-\mathcal{J}^K(u^K_{\cdot})\\
        &\le \mathbb{E}\Big\{\int_{0}^{T}\Big[\big\langle -H_x(t, x^K_{t}, u^K_{t}, \alpha_t, q^K_{t}, z^K_{t};\mathcal{Q}^K_{t,\alpha_t},\mathcal{S}^K_{t,\alpha_t},\mathcal{R}^K_{t,\alpha_t}),x^K_{t}-x^{K*}_{t} \big\rangle\\
        & \quad+\big\langle (A_{t,\alpha_t}x^K_{t}+B_{t,\alpha_t}u^K_{t})^{\top}-(A_{t,\alpha_t}x^{K*}_{t}+B_{t,\alpha_t}u^{K*}_{t})^{\top}, q^K_{t} \big\rangle\\
        & \quad+\big\langle (C_{t,\alpha_t}x^K_{t}+D_{t,\alpha_t}u^K_{t})^{\top}-(C_{t,\alpha_t}x^{K*}_{t}D_{t,\alpha_t}u^{K*}_{t})^{\top}, z^K_{t} \big\rangle \Big]dt\\
        & \quad-\big\langle \frac{1}{2}[x^{\top}_T \mathcal{G}^K_{T,\alpha_T}x_T+x^{\top}_0 K_{0, \alpha_0}x_0], x^K_{T}-x^{K*}_{T} \big\rangle |x_0=\xi, \alpha_0=i_0\Big\}=0.
    \end{aligned}
    \nonumber
\end{equation}
Consequently, $\mathcal{J}^K(u^{K*}_{\cdot})-\mathcal{J}^K(u^K_{\cdot})\le 0$, which implies $u^*_{\cdot}$ is an optimal control of Problem (SLQ-MRS-UC). 
Moreover, $(x^*_{\cdot},u^*_{\cdot})$ is the unique optimal pair of Problem (SLQ-MRS-UC). 
    \end{proof}

\subsection{\texorpdfstring{Stochastic Hamiltonian System with Markov \\Regime-Switching for Constrainted Control}{Stochastic Hamiltonian System with Markov Regime-Switching for Constrainted Control}}\label{sec: Hamiltonian.2}
In this subsection, we address Problem (SLQ-MRS-CC) by developing generalized stochastic Hamiltonian systems, 
thereby characterizing optimal open-loop strategies for non-negative constrained cases. 

We introduce a generalized stochastic Hamiltonian system for Problem (SLQ-MRS-CC) as follows, for $t\in [0,T]$, 
\begin{equation}\label{Markov_eq26}
    \begin{cases}
        \begin{aligned}
            dx^*_t=&[A_{t,\alpha_t}x^*_t+B_{t,\alpha_t}u^*_t]dt+[C_{t,\alpha_t}x^*_t+D_{t,\alpha_t}u^*_t]dW_t,\\
        \end{aligned} \\
        \begin{aligned}
            -dq^*_t=&[A_{t,\alpha_t}q^*_t+C_{t,\alpha_t}z^*_t+Q_{t,\alpha_t}x^*_t+S_{t,\alpha_t}u^*_t]dt-z^*_tdW_t-\eta^*_t\bullet d\widetilde{M}_t, \\
        \end{aligned} \\
            x^*_0=\xi ,~\alpha_0=i_0,~q^*_T=G_{T,\alpha_T}x^*_T,\\
           \end{cases}
\end{equation}
and
\begin{equation}
 \begin{aligned}\label{eq16g}
           &H(t, x^*_{t}, u^*_{t}, \alpha_t, q^*_{t}, z^*_{t};Q_{t,\alpha_t},S_{t,\alpha_t},R_{t,\alpha_t})\\
           &\ \ \quad=\max_ {u\in\mathbb{R}^{m}_+}H(t, x^*_{t}, u, \alpha_t, q^*_{t},  z^*_{t};Q_{t,\alpha_t},S_{t,\alpha_t},R_{t,\alpha_t}).
        \end{aligned}\\
\end{equation} 
\begin{theorem}\label{Thm Hamiltonian.2.1}
    If there exists a relaxed compensator $K_{\cdot, \alpha_{\cdot}}\in\varUpsilon $, 
    the  generalized stochastic Hamiltonian system \eqref{Markov_eq26}-\eqref{eq16g} of Problem (SLQ-MRS-CC) admits a unique solution 
    $(x^*_{\cdot}, u^*_{\cdot}, q^*_{\cdot}, z^*_{\cdot},\eta^* _{\cdot})\in L^2_{\mathbb{F}}(0, T; \mathbb{R}^n)\times \mathcal{U}^c_{ad}
    \times L^2_{\mathbb{F}}(0, T; \mathbb{R}^n)\times L^2_{\mathbb{F}}(0, T; \mathbb{R}^n)
    \times\mathscr{M}^2_{\mathbb{F}}(0, T; \mathbb{R}^n)$, where $(x^*_{\cdot},u^*_{\cdot})$ is the unique optimal pair for Problem (SLQ-MRS-CC). 
\end{theorem}
The proof is similar to the Theorem \ref{Thm 3.1}, we omit it.

One of the  key contributions of this paper is to enable the solution of a class of Markov regime-switching diffusion-based FBSDEs 
with control constraints under broader conditions, especially in cases where the conventional monotonicity condition is violated.
We examine the generalized stochastic Hamiltonian system of Problem (SLQ-MRS-CC) with $R_{t,\alpha_t}<{\bf{0}}$ and $R_{t,\alpha_t}={\bf{0}}$. 

When $R_{t,\alpha_t}<{\bf{0}}$, two cases need to be discussed:(i)If $B^{\top}_{t, \alpha_t}q^*_t+D^{\top}_{t, \alpha_t}z^*_t+S^{\top}_{t, \alpha_t}x^*_{t}>{\bf{0}}$, 
then the optimal control is 
$u^*_t=-R_{t, \alpha_t}^{-1}\bigl[B^{\top}_{t, \alpha_t}q^*_t+D^{\top}_{t, \alpha_t}z^*_t+S^{\top}_{t, \alpha_t}x^*_{t}\bigr]. $
Thus, (\ref{Markov_eq26}) 
can be rewritten as a fully coupled linear FBSDE that admits a unique solution from Theorem \ref{Thm Hamiltonian.2.1}, but
does not satisfy the classic monotonicity condition.
(ii) If $B^{\top}_{t, \alpha_t}q^*_t+D^{\top}_{t, \alpha_t}z^*_t+S^{\top}_{t, \alpha_t}x^*_{t}\le {\bf{0}}$, 
then the optimal control is $u^*_t={\bf{0}}$. 
Thus, the generalized stochastic Hamiltonian system of Problem (SLQ-MRS-CC) can be rewritten as the following FBSDE: 
\begin{equation}\label{Markov 2FBSDE}
    \begin{cases}
        \begin{aligned}
            &dx^*_t=A_{t,\alpha_t}x^*_tdt+C_{t,\alpha_t}x^*_tdW_t,~~~t\in [0,T], \\
            &-dq^*_t=[A_{t,\alpha_t}q^*_t+C_{t,\alpha_t}z^*_t+Q_{t,\alpha_t}x^*_t]dt-z^*_tdW_t-\eta^*_t\bullet d\widetilde{M}_t,~~~t\in [0,T], \\
            &x^*_0=\xi ,~\alpha_0=i_0,~q^*_T=G_{T,\alpha_T}x^*_T.\\
        \end{aligned}
    \end{cases}
\end{equation}
From Theorem \ref{Thm Hamiltonian.2.1}, 
\eqref{Markov 2FBSDE} admits a unique solution.

When $R_{t,\alpha_t}={\bf{0}}$, the generalized stochastic Hamiltonian system of Problem (SLQ-MRS-CC) is not a classical FBSDE. 
From Theorem \ref{Thm Hamiltonian.2.1}, the generalized stochastic Hamiltonian system of Problem (SLQ-MRS-CC) admits a unique solution 
in this case and the optimal control is $u^*_t={\bf{0}}$. 

\section{Riccati Equation with Markov Regime-Switching}\label{sec: Riccati}
This section investigates Riccati equations for indefinite SLQ control problems with Markov regime-switching. 
\subsection{Riccati Equation with Markov Regime-Switching for Unconstrained Control}\label{sec: Riccati.1}
We establish the relationship between corresponding Hamiltonian system 
and Riccati equation for Problem (SLQ-MRS-UC).
By using the relaxed compensator method, 
we obtain the existence and uniqueness of 
the solution to the corresponding Riccati equation. 
Based on this, we derived the closed-loop optimal 
control form of Problem (SLQ-MRS-UC).

Riccati equation of Problem (SLQ-MRS-UC) is as follows
\begin{equation}\label{Markov_eq25}
    \begin{cases}
        \begin{aligned}
        &{\dot P}_{t, \alpha_t}+P_{t, \alpha_t}A_{t, \alpha_t}+A^{\top}_{t, \alpha_t}P_{t, \alpha_t}+C^{\top}_{t, \alpha_t}P_{t, \alpha_t}C_{t, \alpha_t}+Q_{t, \alpha_t}+\sum_{j=1}^{d}\pi_{\alpha_t,j}P_{t, j}\\
        &\quad -[P_{t, \alpha_t}B_{t, \alpha_t}+S_{t, \alpha_t}+C^{\top}_{t, \alpha_t}P_{t, \alpha_t}D_{t, \alpha_t}][R_{t, \alpha_t}+D^{\top}_{t, \alpha_t}P_{t, \alpha_t}D_{t, \alpha_t}]^{-1}\\
        &\ \qquad [B^{\top}_{t, \alpha_t}P_{t, \alpha_t}+D^{\top}_{t, \alpha_t}P_{t, \alpha_t}C_{t, \alpha_t}+S^{\top}_{t, \alpha_t}]={\textbf{0}},~~~(t,\alpha_t)\in [0,T]\times\mathbb{S}, \\
        \end{aligned}\\
        R_{t, \alpha_t}+D^{\top}_{t, \alpha_t}P_{t, \alpha_t}D_{t, \alpha_t}>{\textbf{0}},~~~(t,\alpha_t)\in [0,T]\times\mathbb{S},\\
        P_{T, \alpha_T}=G_{T, \alpha_T}.\\
    \end{cases}
\end{equation} 
\begin{theorem}\label{Thm Riccati.2.1}
    If there exists a relaxed compensator $K_{\cdot, \alpha_{\cdot}}\in \varUpsilon $, then Riccati equation \eqref{Markov_eq25} of Problem (SLQ-MRS-UC) admits a unique solution 
    $P_{\cdot, \alpha_{\cdot}}\in C([0,T]\times \mathbb{S};\mathcal{S}^n)$. Moreover, $P_{\cdot, \alpha_{\cdot}}\ge K_{\cdot, \alpha_{\cdot}}$ $a.e.~$ for $(t, \alpha_t)\in [0,T]\times \mathbb{S}$. 
    Additionally, a unique closed-loop optimal control $u^{*}_{\cdot}$ of Problem (SLQ-MRS-UC) has the following form   
    \begin{equation}\label{Markov_eq28}
            \begin{aligned}
            u^*_t=&-\bigl[R_{t, \alpha_t}+D^{\top}_{t, \alpha_t}P_{t, \alpha_t}D_{t, \alpha_t}\bigr]^{-1}\bigl[B^{\top}_{t, \alpha_t}P_{t, \alpha_t}+D^{\top}_{t, \alpha_t}P_{t, \alpha_t}C_{t, \alpha_t}+S^{\top}_{t, \alpha_t}\bigr]x^*_t, \\
        \end{aligned}
    \end{equation}
    where $x^{*}_{\cdot}$ satisfies  (\ref{Markov_eq1}) with the closed-loop control $u^*_t$ in (\ref{Markov_eq28}).
    \end{theorem}
    \begin{proof} 
        For a given $K_{\cdot, \alpha_{\cdot}}\in\varUpsilon $, 
we introduce the following associated Riccati equation, 
        \begin{equation}\label{Markov_eq27}
            \begin{cases}
                \begin{aligned}
                &{\dot P}^K_{t, \alpha_t}+P^K_{t, \alpha_t}A_{t, \alpha_t}+A^{\top}_{t, \alpha_t}P^K_{t, \alpha_t}+C^{\top}_{t, \alpha_t}P^K_{t, \alpha_t}C_{t, \alpha_t}+\mathcal{Q}^K_{t, \alpha_t}+\sum_{j=1}^{d}\pi_{\alpha_t,j}P^K_{t, j}\\
                &\quad -[P^K_{t, \alpha_t}B_{t, \alpha_t}+\mathcal{S}^K_{t, \alpha_t}+C^{\top}_{t, \alpha_t}P^K_{t, \alpha_t}D_{t, \alpha_t}][\mathcal{R}^K_{t, \alpha_t}+D^{\top}_{t, \alpha_t}P^K_{t, \alpha_t}D_{t, \alpha_t}]^{-1}\\
                &\ \qquad[B^{\top}_{t, \alpha_t}P^K_{t, \alpha_t}+D^{\top}_{t, \alpha_t}P^K_{t, \alpha_t}C_{t, \alpha_t}+{(\mathcal{S}^K_{t, \alpha_t})}^{\top}]={\bf{0}},~~~(t,\alpha_t)\in [0,T]\times\mathbb{S}, \\
                \end{aligned}\\
                \mathcal{R}^K_{t, \alpha_t}+D^{\top}_{t, \alpha_t}P^K_{t, \alpha_t}D_{t, \alpha_t}>{\bf{0}},~~~(t,\alpha_t)\in [0,T]\times\mathbb{S},\\
                P^K_{T, \alpha_T}=\mathcal{G}^K_{T, \alpha_T}.\\
            \end{cases}
        \end{equation}
        The equation $\eqref{Markov_eq27}$ admits a unique solution 
        $P^K_{\cdot, \alpha_{\cdot}}\in C([0,T]\times \mathbb{S};\mathcal{S}^n)$, see Zhang {\em et al.} \cite{zhang2021open}. 
        Moreover, $P^K_{\cdot, \alpha_{\cdot}}$ is uniformly bounded and non-negative. 
        In fact, it can be verified that the solution $P_{t, \alpha_t}$ of \eqref{Markov_eq25} and $P^K_{t, \alpha_t}$ of $\eqref{Markov_eq27}$ have the relationship 
        $P_{t, \alpha_t}=P^K_{t, \alpha_t}+K_{t, \alpha_t}.$
        Thus, Riccati equation \eqref{Markov_eq25} of Problem (SLQ-MRS-UC) admits a unique solution 
        $P_{\cdot, \alpha_{\cdot}}\in C([0,T]\times \mathbb{S};\mathcal{S}^n)$ and 
        $P_{\cdot, \alpha_{\cdot}}\ge K_{\cdot, \alpha_{\cdot}}$ $a.e.$ for $(t, \alpha_t)\in [0,T]\times \mathbb{S}$. 
        Additionally, applying It$\hat{{\rm o}}$'s formula to $P^K_{t, \alpha_t}x^{K*}_t$, and comparing it with $q^{K*}_t$, 
        we obtain the following relationships  
        \begin{equation}\label{Markov_eq29}
            \begin{aligned}
                &q^{K*}_t=P^K_{t, \alpha_t}x^{K*}_t,\eta^{K*}_{i,j}=(P^K_{t, j}-P^K_{t, i})x^{K*}_t,\\
                &z^{K*}_t=P^K_{t, \alpha_t}(C_{t, \alpha_t}x^{K*}_t+D_{t, \alpha_t}u^{K*}_t).
            \end{aligned}
            \end{equation}
            Substituting $\eqref{Markov_eq29}$ into $\eqref{Markov_eq13}$, 
            we can obtain the optimal control $u^{K*}_t$ with the closed-loop form for the system described by equations $\eqref{Markov_eq1}$ and $\eqref{Markov_eq7}$
            \begin{equation}\label{Markov_eq30}
                \begin{aligned}
                    u^{K*}_t=&-\bigl[\mathcal{R}^K_{t, \alpha_t}+D^{\top}_{t, \alpha_t}P^K_{t, \alpha_t}D_{t, \alpha_t}\bigr]^{-1}\bigl[B^{\top}_{t, \alpha_t}P^K_{t, \alpha_t}+D^{\top}_{t, \alpha_t}P^K_{t, \alpha_t}C_{t, \alpha_t}+{(\mathcal{S}^K_{t, \alpha_t})}^{\top}\bigr]x^{K*}_t, 
                \end{aligned}
                \end{equation}
                where $x^{K*}_t$ satisfies   (\ref{Markov_eq1}) with the closed-loop control $u^{K*}_t$ in (\ref{Markov_eq30}).
   
 
                According to Theorem \ref{Thm 3.1}, $(x^*_t,u^*_t)=(x^{K*}_t,u^{K*}_t)$ is also the unique optimal pair of Problem (SLQ-MRS-UC).
                Additionally, it is easy to verify that 
$$\begin{aligned}
    &\mathcal{R}^K_{t, \alpha_t}+D^{\top}_{t, \alpha_t}P^{K}_{t, \alpha_t}D_{t, \alpha_t}=R_{t, \alpha_t}+D^{\top}_{t, \alpha_t}P_{t, \alpha_t}D_{t, \alpha_t},\\  
    &B^{\top}_{t, \alpha_t}P^K_{t, \alpha_t}+D^{\top}_{t, \alpha_t}P^K_{t, \alpha_t}C_{t, \alpha_t}+(\mathcal{S}^K_{t, \alpha_t})^{\top}=B^{\top}_{t, \alpha_t}P_{t, \alpha_t}+D^{\top}_{t, \alpha_t}P_{t, \alpha_t}C_{t, \alpha_t}+S^{\top}_{t, \alpha_t}.\\ 
\end{aligned}$$
    \end{proof}
In particular, we find that
$$\begin{aligned}
                &\mathcal{Q}_{t, \alpha_t}(P)=Q_{t, \alpha_t}+{\dot P}_{t, \alpha_t}+P_{t, \alpha_t}A_{t, \alpha_t}+A^{\top}_{t, \alpha_t}P_{t, \alpha_t}\\
                &\ \ \qquad\qquad+C^{\top}_{t, \alpha_t}P_{t, \alpha_t}C_{t, \alpha_t}+\sum_{j=1}^{d}\pi_{\alpha_{t_-},j}(P_{t, j}-P_{t, \alpha_{t_-}}), \\
                &\mathcal{S}_{t, \alpha_t}(P)=S_{t, \alpha_t}+P_{t, \alpha_t}B_{t, \alpha_t}+C^{\top}_{t, \alpha_t}P_{t, \alpha_t}D_{t, \alpha_t},  \\
                &\mathcal{R}_{t, \alpha_t}(P)=R_{t, \alpha_t}+D^{\top}_{t, \alpha_t}P_{t, \alpha_t}D_{t, \alpha_t},  \\
                &\mathcal{G}_{T, \alpha_T}(P)=G_{T, \alpha_T}-P_{T, \alpha_T}
\end{aligned}$$
satisfy Assumption \ref{Asm 2.1},
which means that the solution of Riccati equation is a special relaxed compensator. 

Theoretically distinct from Li and Zhou \cite{li2002indefinite}, 
this study innovatively introduces a relaxed compensator. 
Our approach breaks through the traditional limitation of directly solving 
CGREs. By constructing an equivalent transformation between indefinite and positive definite problems, 
we establish the key correspondence: the equivalence of Riccati equations 
under indefinite versus positive definite conditions. 
This enables direct application of mature positive-definite theories to indefinite problems. 
Computationally, this method significantly reduces complexity, 
providing a new theoretical framework for solving such control problems. 

\subsection{Riccati Equation with Markov Regime-Switching for Constrained Control}\label{sec: Riccati.2}
In this subsection, we focus on the case where 
the state variable $x_\cdot$ is scalar.
To obtain the closed-loop optimal control of Problem (SLQ-MRS-CC), we need to define the following mappings:
\begin{equation}
    \begin{aligned}
       &H_1(t, \upsilon, P^1_{t, \alpha_t}, \alpha_t;R_{t, \alpha_t},S_{t, \alpha_t})\\
       &:=\upsilon^{\top}(D^{\top}_{t, \alpha_t}D_{t, \alpha_t}P^1_{t, \alpha_t}+R_{t, \alpha_t})\upsilon+2\upsilon^{\top}(B^{\top}_{t, \alpha_t}P^1_{t, \alpha_t}+D^{\top}_{t, \alpha_t}P^1_{t, \alpha_t}C_{t, \alpha_t}+S^{\top}_{t, \alpha_t}),\\
       &H_2(t, \upsilon, P^2_{t, \alpha_t}, \alpha_t;R_{t, \alpha_t},S_{t, \alpha_t})\\
       &:=\upsilon^{\top}(D^{\top}_{t, \alpha_t}D_{t, \alpha_t}P^2_{t, \alpha_t}+R_{t, \alpha_t})\upsilon-2\upsilon^{\top}(B^{\top}_{t, \alpha_t}P^2_{t, \alpha_t}+D^{\top}_{t, \alpha_t}P^2_{t, \alpha_t}C_{t, \alpha_t}+S^{\top}_{t, \alpha_t}),\\
    \end{aligned}
    \nonumber
\end{equation}
and, for $l=1,2$ and $(t, P^l, \alpha_t)\in [0,T]\times \mathbb{R}\times \mathbb{S}$, 
\begin{equation}
    \begin{aligned}
       &\mathcal{H}_l(t, P^l_{t, \alpha_t}, \alpha_t;R_{t, \alpha_t},S_{t, \alpha_t}):=\inf_{\upsilon \in \varGamma} H_l(t, \upsilon, P^l_{t, \alpha_t}, \alpha_t;R_{t, \alpha_t},S_{t, \alpha_t}).
    \end{aligned}
    \nonumber
\end{equation}
GSRE for Problem (SLQ-MRS-CC) is as follows
\begin{equation}\label{Initial Markov_eq34}
    \begin{cases}
        \begin{aligned}
        &{\dot P}^l_{t, \alpha_t}+(2A_{t, \alpha_t}+C^{\top}_{t, \alpha_t}C_{t, \alpha_t})P^l_{t, \alpha_t}+Q_{t, \alpha_t}+\sum_{j=1}^{d}\pi_{\alpha_t, j}P^l_{t, j}\\
        &\quad +\mathcal{H}_l(t, P^l_{t, \alpha_t}, \alpha_t;R_{t, \alpha_t},S_{t, \alpha_t})=0,~~~(t,\alpha_t)\in [0,T]\times\mathbb{S},\\
        \end{aligned}\\
        P^l_{T, \alpha_T}=G_{T, \alpha_T},\\
        R_{t, \alpha_t}+D^{\top}_{t, \alpha_t}D_{t, \alpha_t}P^l_{t, \alpha_t}>{\bf{0}},~~~(t,\alpha_t)\in [0,T]\times\mathbb{S},\\
    \end{cases} 
\end{equation}
where $l=1,2$.

Then, we introduce the following GSRE associated with a given $K_{\cdot, \alpha_{\cdot}}\in\varUpsilon $,
\begin{equation}\label{Markov_eq34}
    \begin{cases}
        \begin{aligned}
        &{\dot P}^{K, l}_{t, \alpha_t}+(2A_{t, \alpha_t}+C^{\top}_{t, \alpha_t}C_{t, \alpha_t})P^{K, l}_{t, \alpha_t}+\mathcal{Q}^K_{t, \alpha_t}+\sum_{j=1}^{d}\pi_{\alpha_t, j}P^{K, l}_{t, j}\\
        &\quad +\mathcal{H}_l(t, P^{K, l}_{t, \alpha_t}, \alpha_t;\mathcal{R}^K_{t, \alpha_t},\mathcal{S}^K_{t, \alpha_t})=0,~~~(t,\alpha_t)\in [0,T]\times\mathbb{S},\\
        \end{aligned}\\
        P^{K, l}_{T, \alpha_T}=\mathcal{G}^K_{T, \alpha_T},\\
        \mathcal{R}^K_{t, \alpha_t}+D^{\top}_{t, \alpha_t}D_{t, \alpha_t}P^{K, l}_{t, \alpha_t}>{\bf{0}},~~~(t,\alpha_t)\in [0,T]\times\mathbb{S},\\
    \end{cases} 
\end{equation}
where $l=1,2$.
\begin{remark}\label{rem Riccati.2.2}
    If $\varGamma$ is symmetric, namely, $-\upsilon \in \varGamma$ whenever $\upsilon \in \varGamma$, 
    then $\mathcal{H}_1(t, P^{K, 1}_{t, \alpha_t},\alpha_t;\\\mathcal{R}^K_{t, \alpha_t},\mathcal{S}^K_{t, \alpha_t})=\mathcal{H}_2(t, P^{K, 2}_{t, \alpha_t}, \alpha_t;\mathcal{R}^K_{t, \alpha_t},\mathcal{S}^K_{t, \alpha_t})$. 
    If \eqref{Markov_eq34} 
    admits a unique solution, then $P^{K, 1}_{t, \alpha_t}=P^{K, 2}_{t, \alpha_t}$, where both $P^{K, 1}_{t, \alpha_t}$ and $P^{K, 2}_{t, \alpha_t}$ are formulated in terms of $P^K_{t, \alpha_t}$ for notational simplicity. 
    Particularly, when there is no control constraint, 
    i.e. $\varGamma=\mathbb{R}^m$, then  
    \eqref{Markov_eq34} reduces to 
    the normal stochastic Riccati equation of Problem (SLQ-MRS-UC).
\end{remark}
Equation \eqref{Markov_eq34} is referred to as the {\em generalized stochastic Riccati equation} (GSRE). 
\begin{definition}\label{def Riccati.2.1}
    $P^{K, l}_{\cdot, \alpha_{\cdot}}\in C([0,T];\mathbb{R})$ for $l=1,2$ and all $\alpha_{\cdot}\in \mathbb{S}$ is called solution to \eqref{Markov_eq34} 
    if they satisfy \eqref{Markov_eq34}. 
\end{definition}
Based on the previous theoretical proof results in this paper, we can verify the following similar relationship: 
$P^l_{t, \alpha_t}=P^{K, l}_{t, \alpha_t}+K_{t, \alpha_t},~~~l=1,2.$
Consequently, the existence and uniqueness of the solution to GSRE for Problem (SLQ-MRS-CC) is equivalent to 
that of \eqref{Markov_eq34}. 

To find the optimal feedback control for Problem (SLQ-MRS-CC), 
we also need to define 
\begin{equation}
    \begin{aligned}
    &\hat{\upsilon}_l(t, P^{K, l}_{t, \alpha_t}, \alpha_t;\mathcal{R}^K_{t, \alpha_t},\mathcal{S}^K_{t, \alpha_t})=\arg\min\limits_{\upsilon \in \varGamma}H_l(t, \upsilon, P^{K, l}_{t, \alpha_t}, \alpha_t;\mathcal{R}^K_{t, \alpha_t},\mathcal{S}^K_{t, \alpha_t}),~~~l=1,2.
    \end{aligned}
    \nonumber
\end{equation}
 \begin{remark}\label{rem Riccati.2.3}
    If $\varGamma$ is symmetric, then $\hat{\upsilon}_1(t, P^K_{t, \alpha_t}, \alpha_t;\mathcal{R}^K_{t, \alpha_t},\mathcal{S}^K_{t, \alpha_t})=-\hat{\upsilon}_2(t, P^K_{t, \alpha_t},\alpha_t;\\\mathcal{R}^K_{t, \alpha_t},\mathcal{S}^K_{t, \alpha_t})$. 
    Particularly, if $\varGamma=\mathbb{R}^m$, then 
    \begin{equation}
    \begin{aligned}
    \hat{\upsilon}_1(t, P^K_{t, \alpha_t}, \alpha_t;\mathcal{R}^K_{t, \alpha_t},\mathcal{S}^K_{t, \alpha_t})=&-(\mathcal{R}^K_{t, \alpha_t}+D^{\top}_{t, \alpha_t}D_{t, \alpha_t}P^K_{t, \alpha_t})^{-1}(B^{\top}_{t, \alpha_t}P^K_{t, \alpha_t}\\
    &+D^{\top}_{t, \alpha_t}P^K_{t, \alpha_t}C_{t, \alpha_t}+(\mathcal{S}^K_{t, \alpha_t})^{\top}).
    \end{aligned}
    \nonumber
\end{equation}
\end{remark}
\begin{theorem}\label{Thm Riccati.2.3}
For a given relaxed compensator $K_{\cdot, \alpha_{\cdot}}\in\varUpsilon $ 
of Problem (SLQ-MRS-CC), \eqref{Initial Markov_eq34} admits unique positive solution
$P^l_{\cdot, \alpha_{\cdot}}$ for $l=1,2$ and all $\alpha_{\cdot}\in \mathbb{S}$.  
The optimal feedback control of Problem (SLQ-MRS-CC)
at time $t$ and state $x_t$ is
$u^*_t=\hat{\upsilon}_1(t, P^1_{t, \alpha_t}, \alpha_t;R_{t, \alpha_t},S_{t, \alpha_t})x^{+}_t+\hat{\upsilon}_2(t, P^2_{t, \alpha_t}, \alpha_t;R_{t, \alpha_t},S_{t, \alpha_t})x^{-}_t.
$
\end{theorem}
\begin{proof}
 Hu {\em et al.} \cite{hu2022constrained} 
provide a complete proof for similar equations 
like \eqref{Markov_eq34}. 
For \eqref{Markov_eq34}, we construct an auxiliary linear system of 
ordinary differential equations and use the comparison theorem to prove the uniform boundedness 
of the solution. Then, we apply the Lipschitz approximation technique 
and the monotone convergence principle to 
prove the existence of the solution. 
The uniqueness analysis is accomplished 
by using the Lipschitz condition and Gronwall's inequality.
Combining the properties of the control constraint cone, 
we further derived the existence and uniqueness 
of the optimal control of Problem (SLQ-MRS-CC) 
and its feedback expression. 
Owing to page limitations, detailed derivations are not repeated in this paper.
\end{proof}
Unlike the studies by Hu and Zhou \cite{hu2005constrained} and Hu {\em et al.} \cite{hu2022constrained}, 
this paper makes two key improvements. 
On the one hand, to handle control constraints, we expand the research boundary 
by incorporating all cost functional coefficient cases into the analytical framework, 
forming a more comprehensive system.
On the other hand, using the constructed relaxed compensator and conclusions from the positive definite case, 
we directly obtain the existence and uniqueness of solutions 
to the corresponding ESREs under the indefinite case without re-proving. 
This simplifies the solution process for constrained indefinite problems, 
lowers the theoretical threshold of repeated complex equation proofs, 
and provides a more efficient, universal framework. 
\section{Equity-Bond Asset Allocation Problem}\label{sec: 5}
In this section, we further investigate Problem (EBAA) mentioned in the Introduction 
and provide some numerical simulations to illustrate our theoretical results. 
 
To solve this kind of Problem (EBAA) and derive a unique optimal control strategy, 
we utilize the theoretical framework developed in this paper 
to avoid direct pseudo-inverse computation and find at least one $K_{\cdot, \alpha_{\cdot}}\in\varUpsilon$ such that the quadruple 
$(\mathcal{Q}^K_{\cdot, \alpha_{\cdot}}, \mathcal{S}^K_{\cdot, \alpha_{\cdot}}, \mathcal{R}^K_{\cdot, \alpha_{\cdot}}, \mathcal{G}^K_{\cdot, \alpha_{\cdot}})$ 
satisfies Assumption \ref{Asm 2.1}.
We model the market environment using a two-state Markov chain
that is $\mathbb{S}=\{1,2\}$ 
and assume that the generator of the Markov chain is  
$\mathbb{G}=\left[-1\ 1;1\ -1\right]$.
The finite time horizon is given with 
    $T = 10$, the initial condition $x_0=0$, 
    and the coefficients of Problem (EBAA) are  
     given by $A_{t,1}=0.03$, $B_{t,1}=0.5$, $D_{t,1}=1$,
            $A_{t,2}=0.07$, $B_{t,2}=0.03$, $D_{t,2}=8$,
$            Q_{t,1}=0.02$, $R_{t,1}=-4$, $G_{T,1}=10$, 
$            Q_{t,2}=0.04$, $R_{t,2}=4$, $G_{T,2}=10$. 
%
    Further assume that $\eta_{i,j}=0$. 
    Let $z_t=8$ if $\alpha_t=1$ 
    and $z_t=1$ if $\alpha_t=2$. 

If $\Phi_{\cdot}\in L^2_{\mathbb{F}}(0, T; \mathbb{R}^{n\times n})$ is the solution to the 
following stochastic differential equation 
\begin{equation}
    \begin{cases}
        d\Phi_\tau=A_{\tau,\alpha_\tau}\Phi_\tau d\tau+C_{\tau,\alpha_\tau}\Phi_\tau dW_\tau,\quad \tau\in [t,T],\\
        \Phi_t=I , \alpha_t=i,
    \end{cases}
    \nonumber
\end{equation}
where $I$ represents the identity matrix of the appropriate dimensions. 
Let $M_{t,\alpha_t}\in \mathcal{S}^n_+$ be a 
symmetric positive semi-definite matrix process. 
Define a modified parameter 
$\Theta ^*_{t,\alpha_t}=\Theta _{t,\alpha_t}+M_{t,\alpha_t}$
Furthermore, $K_{t,\alpha_t}$ also solves the following system 
\begin{equation}\label{new 4}
\begin{cases}
    \begin{aligned}
&K_{t,\alpha_t}A_{t,\alpha_t}+A^{\top}_{t,\alpha_t}K_{t,\alpha_t}+C^{\top}_{t,\alpha_t}K_{t,\alpha_t}C_{t,\alpha_t}+\sum_{j=1}^{d}\pi_{\alpha_{t-},j}(K_{t, j}-K_{t,\alpha_{t-}})\\
&\quad+Q_{t,\alpha_t}+\Theta ^*_{t,\alpha_t}=M_{t,\alpha_t},\quad (t,\alpha_t)\in [0,T]\times\mathbb{S},\\
    \end{aligned}\\
K_{T,\alpha_T}=G_{T,\alpha_T}.\\
\end{cases}
\end{equation}
Applying It$\hat{{\rm o}}$'s formula to $\Phi^\top_\tau K_{\tau,\alpha_\tau}\Phi_\tau$, 
    and integrating from $t$ to $T$, we obtain   
    \begin{equation}
        \begin{aligned}
            &\Phi^\top_T K_{T,\alpha_T}\Phi_T=\Phi^\top_t K_{t,\alpha_t}\Phi_t
            +\int_{t}^{T}\Phi^\top_\tau[K_{\tau,\alpha_\tau}A_{\tau,\alpha_\tau}+A^{\top}_{\tau,\alpha_\tau}K_{\tau,\alpha_\tau}+C^{\top}_{\tau,\alpha_\tau}K_{\tau,\alpha_\tau}C_{\tau,\alpha_\tau}\\
            &\quad+\sum_{i,j=1}^{d}\pi_{i,j}\mathcal{X} _{\{\alpha_{\tau_-}=i\}}(K_{\tau,j}-K_{\tau,i})+\Theta _{\tau,\alpha_\tau}]\Phi_\tau d\tau+\int_{t}^{T}\Phi^\top_\tau[C^{\top}_{\tau,\alpha_\tau}K_{\tau,\alpha_\tau}\\
            &\quad+K_{\tau,\alpha_\tau}C_{\tau,\alpha_\tau}]\Phi_\tau dW_\tau+\int_{t}^{T}\Phi^\top_\tau[\sum_{i,j=1}^{d}\pi_{i,j}\mathcal{X} _{\{\alpha_{\tau_-}=i\}}(K_{\tau,j}-K_{\tau,i})]\Phi_\tau d\widetilde{M}_{i,j}(\tau). 
        \end{aligned}
        \nonumber
    \end{equation}
    Since $K_{T,\alpha_T}=G_{T,\alpha_T}$ and $\Phi_t=I$, 
taking conditional expectation in $\mathcal{F}^\alpha_t$ with respect to $K_{t,\alpha_t}$ and using (\ref{new 4}), 
    we find $K_{t,\alpha_t}$ has the following representation:
$    K_{t,\alpha_t}=\mathbb{E}\bigg[\Phi^\top_T G_{T,\alpha_T}\Phi_T+\int_t^T\Phi^\top_\tau Q_{\tau,\alpha_\tau}\Phi_\tau d\tau|\mathcal{F}^\alpha_t\bigg].
$

    Based on the theoretical results in this paper and the parameter values given above, 
    we can solve $\Phi_{\cdot}$ corresponding to any initial time $t\ (t\in [0,T])$.  
    Using the $\Phi_{\cdot}$ obtained at different initial times, we can further 
    derive the explicit solution for the $K_{t,\alpha_t}$. 
    Taking $t=0$ and $t=3.5$ as examples, the trajectories of $\Phi_{\cdot}$ 
    are shown in Fig.~\ref{Fig 1}(a) and (b).
    By considering all time points $t$, we can obtain all explicit solutions of $K_{t,\alpha_t}$, 
    whose trajectory is shown in Fig.~\ref{Fig 1}(c).
    Notably, the results demonstrate that $K_{t,\alpha_t}$ maintains a lower bound of 
    10 throughout the considered time horizon.

\begin{figure*}
    \centering
    \begin{minipage}{0.32\textwidth}
        \centering
        \includegraphics[height=4cm]{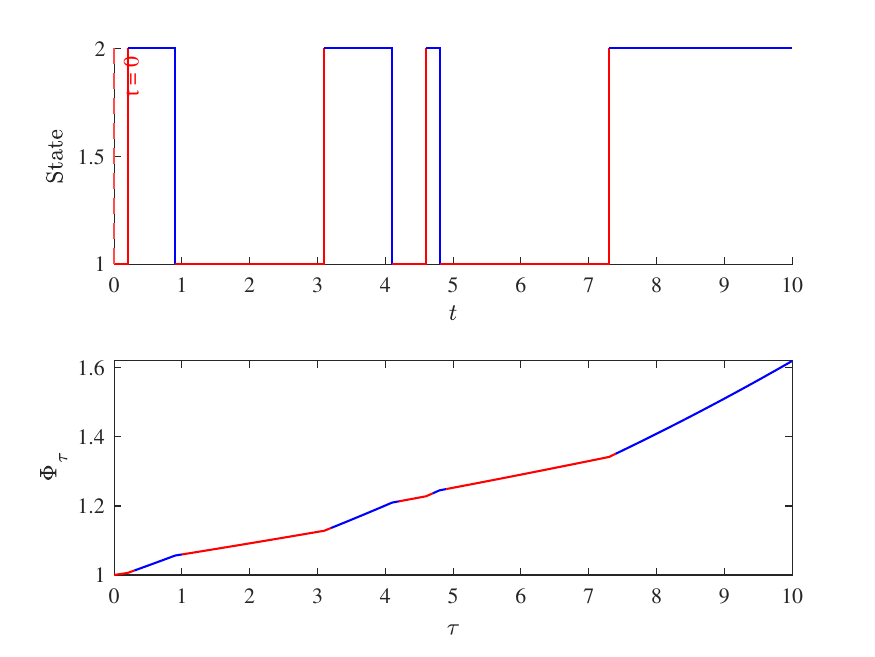}
        \subcaption{$\Phi_\tau$ starts from $\tau=0$ with the corresponding initial market state $\alpha_{0}=1$}
    \end{minipage}
    \hfill
    \begin{minipage}{0.32\textwidth}
        \centering
         \includegraphics[height=4cm]{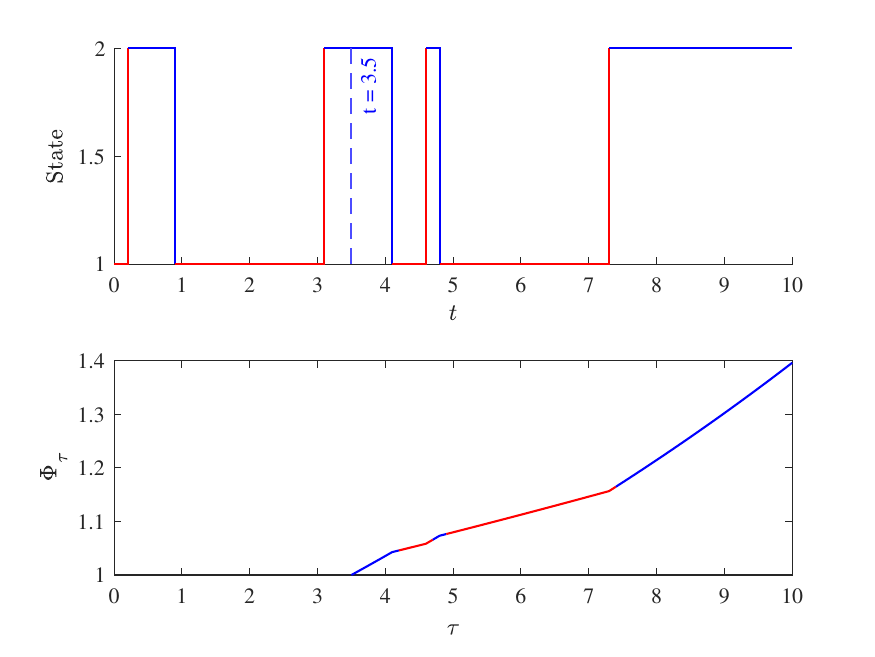}
        \subcaption{$\Phi_\tau$ starts from $\tau=3.5$ with the corresponding initial market state $\alpha_{3.5}=2$}
    \end{minipage}
    \hfill
    \begin{minipage}{0.32\textwidth}
        \centering
    \includegraphics[height=4cm]{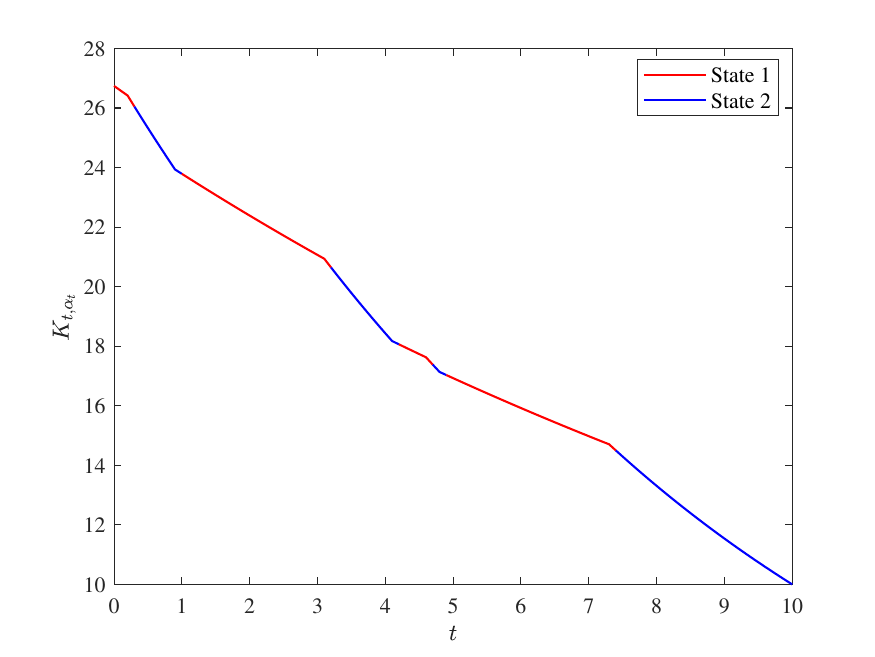}
          \subcaption{The solution of the corresponding $K_{t,\alpha_t}$}
    \end{minipage}

    \caption{The solutions of $\Phi_\tau$ and $K_{t,\alpha_t}$}
    \label{Fig 1}
\end{figure*}

    We get
   $\mathcal{Q}^K_{t,1}=0$, $\mathcal{R}^K_{t,1}\ge 6$,
$\mathcal{G}^K_{T,1}=0$; $\mathcal{Q}^K_{t,2}=0$,
$\mathcal{R}^K_{t,2}\ge 644$, $\mathcal{G}^K_{T,2}=0$. 
%
%
To satisfy the definition of the relaxed compensator, we set $C_{t,\alpha_t}=0$ and 
directly specify the parameter $S_{t,\alpha_t}$ as
$S_{t,\alpha_t}=-K_{t,\alpha_t} B_{t,\alpha_t}$.
Under this specification, the condition
$S_{t,\alpha_t}+K_{t,\alpha_t} B_{t,\alpha_t}+C_{t,\alpha_t}^\top K_{t,\alpha_t} D_{t,\alpha_t}=0$ 
is identically satisfied, which ensures that the obtained $K_{t,\alpha_t}$ fulfills the definition of the relaxed compensator. 
According to Theorem \ref{Thm 2.1}, Problem (EBAA) is well-posed. 

Theorem \ref{Thm 3.1} establishes 
the existence of the unique open-loop optimal control 
for Problem (EBAA-UC), while
Theorem \ref{Thm Riccati.2.1} 
further derives the corresponding unique closed-loop optimal control and its expression is
\begin{equation}\label{Markov_eq39}
    \begin{aligned}
    u^*_t&=-(R_{t, \alpha_t}+D^{\top}_{t, \alpha_t}D_{t, \alpha_t}P_{t, \alpha_t})^{-1}B^{\top}_{t, \alpha_t}P_{t, \alpha_t}x^*_t, \\
\end{aligned}
\end{equation}
where $x^{*}_{\cdot}$ satisfies  (\ref{Markov_eq1}) with the closed-loop control $u^*_t$ in (\ref{Markov_eq39}). 

We present numerical simulations 
to illustrate the theoretical results of Problem (EBAA-UC) 
with Markov regime-switching. 
Here, we adopt a standard method to simulate the Markov chain, that is, 
the state transition of the Markov chain is based on the state transition probability matrix and uniformly distributed random numbers, 
and the next state is determined by the cumulative probability. 
Using the Markov chain $\alpha_t$ with a time step size of $0.1$, 
we obtain the state transition process of the system over the time horizon, as depicted in Fig.~\ref{Fig 3}(a). 
Each vertical line represents a sudden state transition, 
while the horizontal segments indicate the continuation of the state over a specific period. 
By applying the Euler-Maruyama method, along with the Markov chain $\alpha_t$, 
we get the trajectories of $q^*_t$, $x^*_t$, unconstrained open-loop optimal control, $P^*_t$, 
and unconstrained closed-loop optimal control. 
The corresponding dynamics of these variables are illustrated in 
Fig.~\ref{Fig 3}(b)-\ref{Fig 3}(f) respectively. 
Under the unconstrained open-loop optimal control strategy (Fig.~\ref{Fig 3}(d)), 
when the system transitions from State 1 to State 2, 
the instantaneous update of parameters induces a step-like jump 
in the unconstrained open-loop control $u^*_t$. 
In contrast, the unconstrained closed-loop optimal control strategy (Fig.~\ref{Fig 3}(f)) 
dynamically optimizes the control input by solving Riccati 
equation in real-time (Fig.~\ref{Fig 3}(e)) and incorporating feedback 
from the current $x^*_t$ (Fig.~\ref{Fig 3}(c)). 

\begin{figure*}
    \centering
    \begin{minipage}{0.32\textwidth}
        \centering
        \includegraphics[height=4cm]{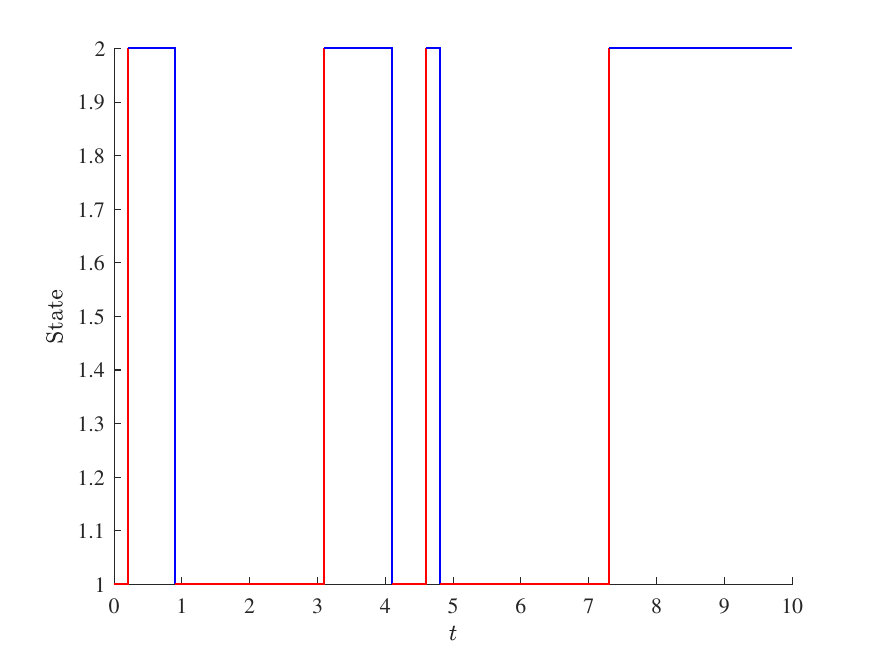}
        \subcaption{Markov regime-switching time series $\alpha_t$}
    \end{minipage}
    \hfill
    \begin{minipage}{0.32\textwidth}
        \centering
        \includegraphics[height=4cm]{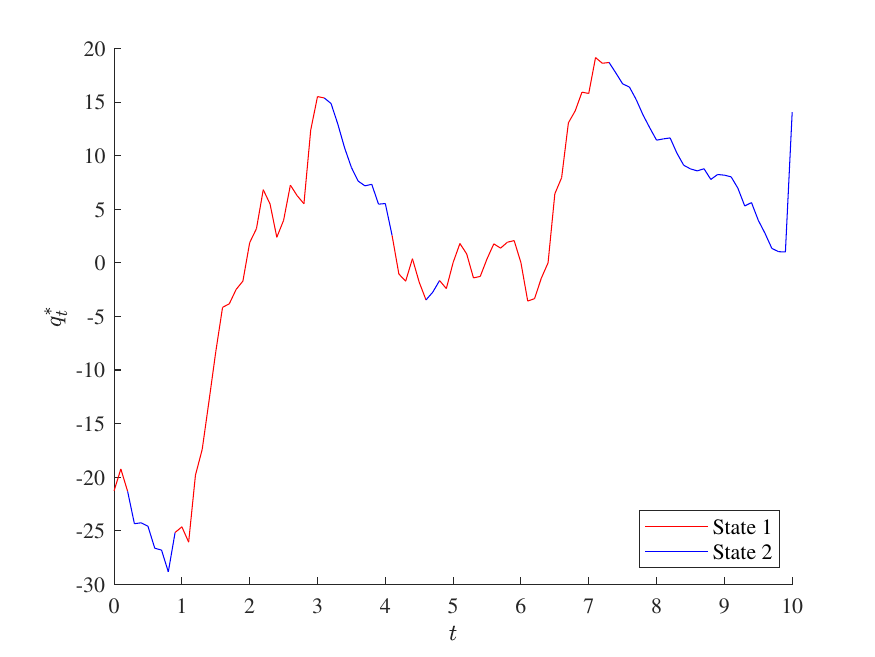}
        \subcaption{Adjoint state $q^*_t$ for unconstrained control}
    \end{minipage}
    \hfill
    \begin{minipage}{0.32\textwidth}
        \centering
        \includegraphics[height=4cm]{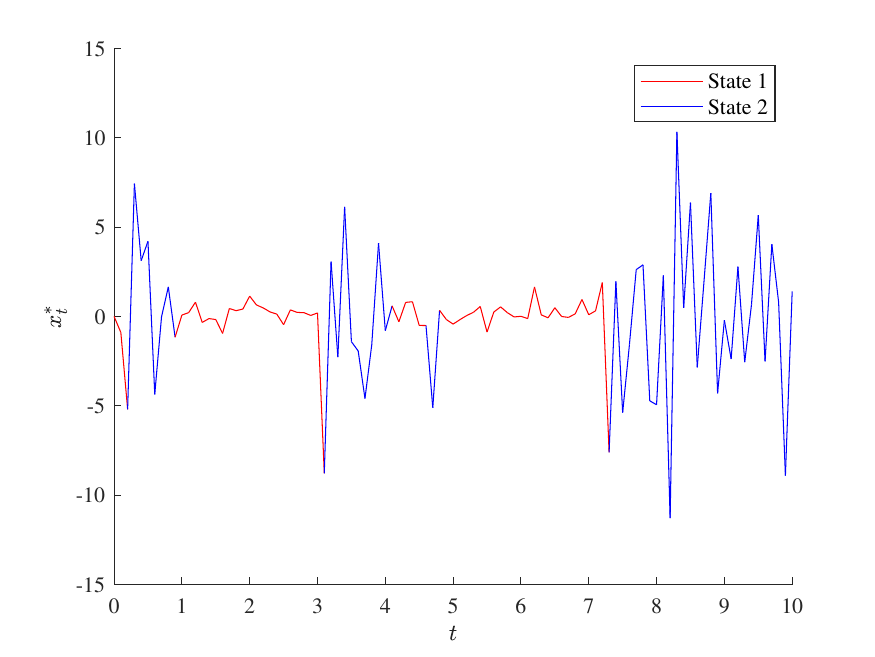}
        \subcaption{The optimal state $x^*_t$}
    \end{minipage}

    \vspace{0.5cm} 

    \begin{minipage}{0.32\textwidth}
        \centering
        \includegraphics[height=4cm]{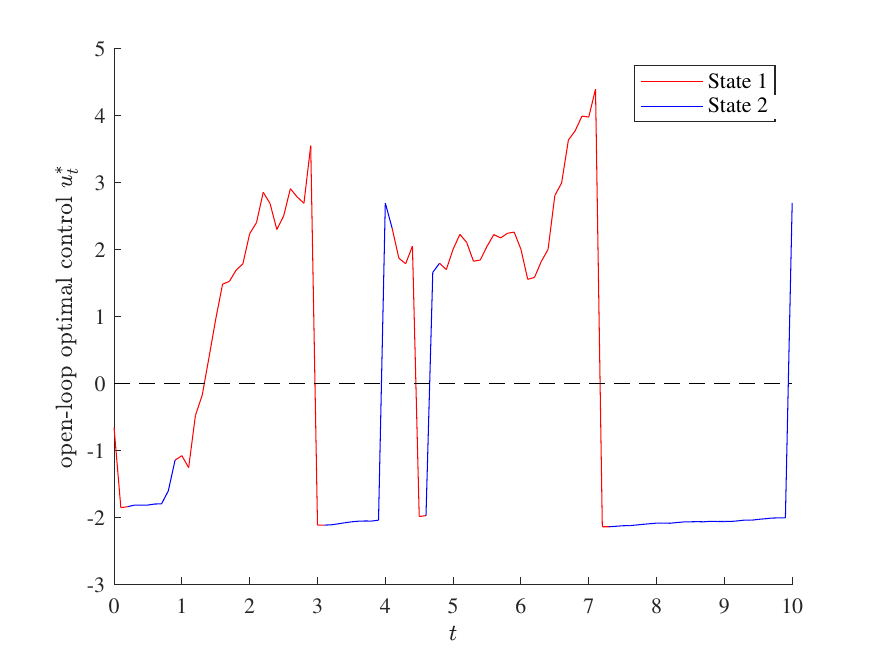}
        \subcaption{Unconstrained open-loop optimal control $u^*_t$}
    \end{minipage}
    \hfill
    \begin{minipage}{0.32\textwidth}
        \centering
        \includegraphics[height=4cm]{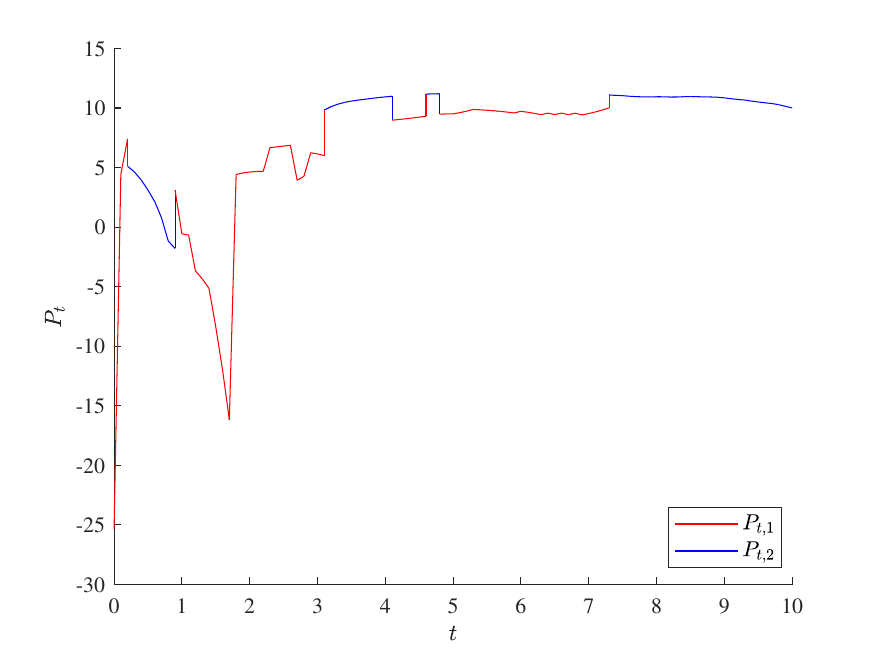}
        \subcaption{The solution of Riccati solution $P_t$ for unconstrained control}
    \end{minipage}
    \hfill
    \begin{minipage}{0.32\textwidth}
        \centering
        \includegraphics[height=4cm]{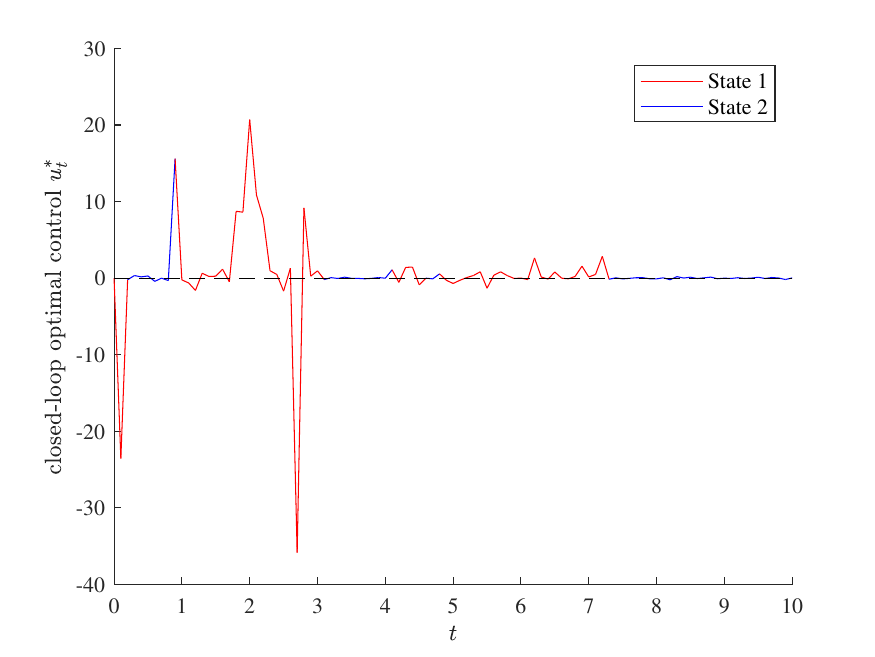}
        \subcaption{Unconstrained closed-loop optimal control $u^*_t$}
    \end{minipage}
    \caption{All simulations under Markov regime-switching framework}
    \label{Fig 3}
\end{figure*}
If Problem (EBAA) prohibits short selling, 
the intensity of stock assets weight adjustment, 
needs to satisfy the hard constraint of $u_t\ge 0$. 
The numerical simulation results of this paper (Fig.~\ref{Fig 3}(d) and \ref{Fig 3}(f)) reveal 
the contradiction between the theoretical assumption and practice. 
Based on the content analyses of Subsection \ref{sec: Hamiltonian.2} and Subsection \ref{sec: Riccati.2}, 
we conduct numerical simulation analyses for Problem (EBAA-CC), 
with the results shown in Fig.~\ref{fig4}. 
Specifically, Fig.~\ref{fig4}(a) and (b)  illustrate the time-varying characteristics of the numerical solutions $P^1_t$ and $P^2_t$ of GSRE, and
Fig.~\ref{fig4}(c) presents the time-varying patterns of the closed-loop optimal control strategy. 

\begin{figure*}
    \centering
    \begin{minipage}{0.32\textwidth}
        \centering
       \includegraphics[height=4cm]{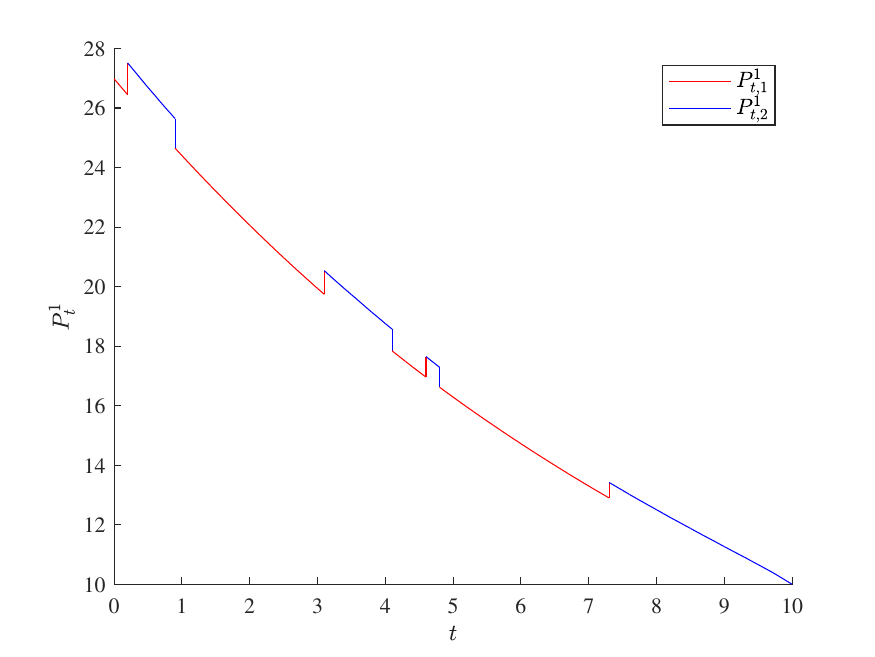}
        \subcaption{$P^1_t$}
    \end{minipage}
    \hfill
    \begin{minipage}{0.32\textwidth}
        \centering
        \includegraphics[height=4cm]{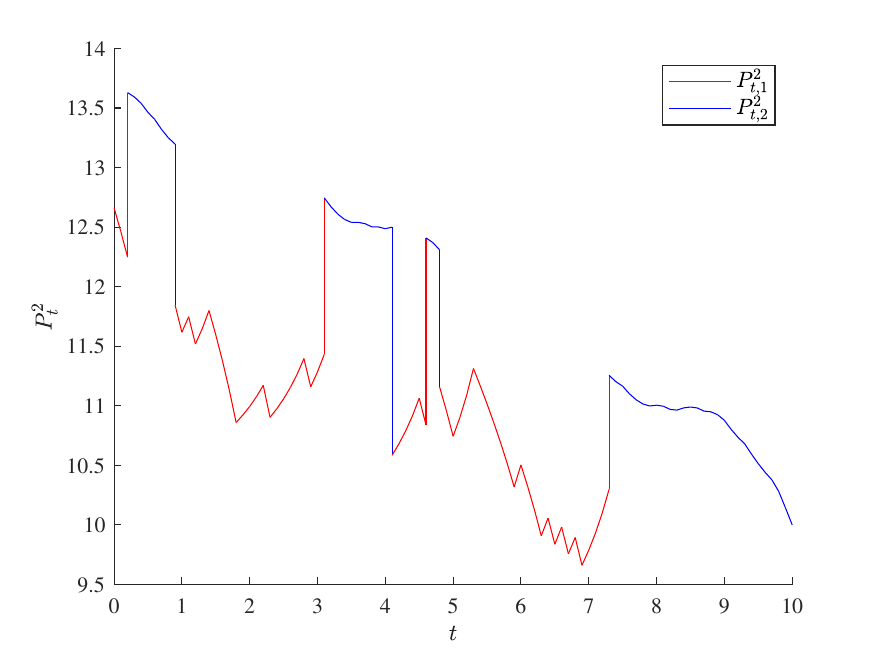}
        \subcaption{$P^2_t$}
    \end{minipage}
    \hfill
    \begin{minipage}{0.32\textwidth}
        \centering
  \includegraphics[height=4cm]{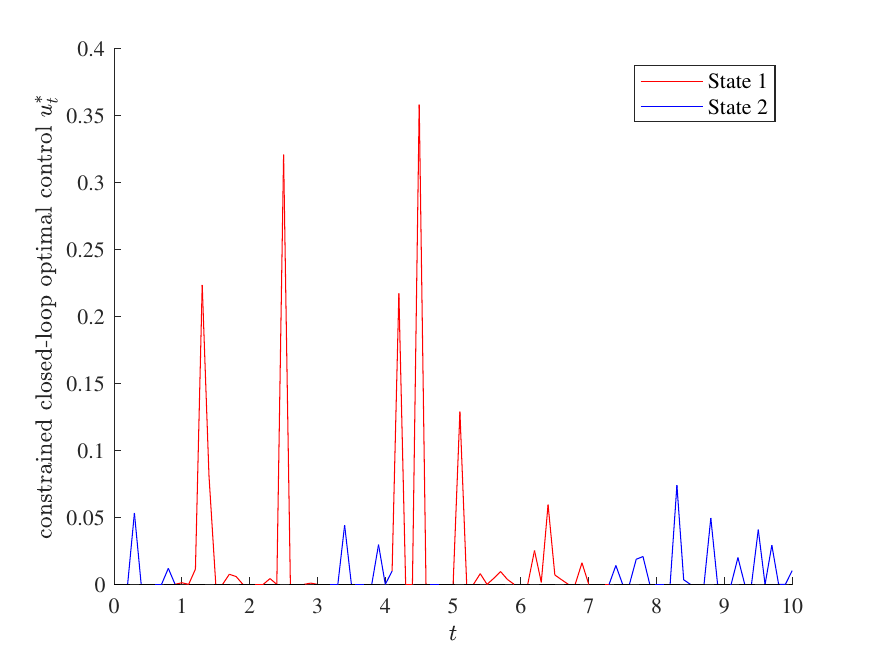}    
\subcaption{Constrained closed-loop optimal control $u^*_t$} 
    \end{minipage}

    \caption{The solutions of GSREs and constrained closed-loop optimal control }
    \label{fig4}
\end{figure*}

\section{Conclusion}\label{sec: 6}
In this paper,an indefinite SLQ optimal control problem with 
Markov regime-switching is studied. 
We define a set of diffusion processes and 
construct relaxed compensators that satisfy these processes.
With such relaxed compensators, 
the analysis is extended from positive definite to indefinite cases, 
establishing the well-posedness of Problem (SLQ-MRS).
We investigate the stochastic Hamiltonian system with Markov regime-switching 
and the corresponding Riccati equation, 
yielding the open-loop and closed-loop optimal controls, respectively. 
Furthermore, we generalize these results to incorporate non-negative control constraints, 
developing generalized stochastic Hamiltonian systems and Riccati equations 
to characterize the optimal strategies.
Applying our framework to Problem  (EBAA), numerical simulations verify the validity of the theoretical results.


\bibliographystyle{elsarticle-harv}
\bibliography{Indefinite-SLQ-MRS}

\end{document}